\documentclass{article}

\usepackage{amsmath,amsthm,bm,colortbl,graphicx,mathrsfs,amssymb,afterpage,cite,float,verbatim,ulem,enumerate}
\usepackage[hyphens]{url}
\RequirePackage[colorlinks,citecolor=blue,urlcolor=blue]{hyperref}

\allowdisplaybreaks
\numberwithin{equation}{section}
\pdfoutput=1

\theoremstyle{plain}
\newtheorem{theorem}{Theorem}[section]

\newtheorem{corollary}[theorem]{Corollary}
\newtheorem{conjecture}[theorem]{Conjecture}
\theoremstyle{definition}
\newtheorem{example}[theorem]{Example}
\theoremstyle{remark}
\newtheorem*{remarks}{Remarks}

\def\eps{\varepsilon}

\newenvironment{newreferences}
               {\section*{References}
                \raggedright
                \begin{list}{}{\setlength{\itemsep}{0pt}
                               \setlength{\parsep}{0pt}
                               \setlength{\labelwidth}{0pt}
                               \setlength{\leftmargin}{12pt}
                               \setlength{\labelsep}{0pt}}
                \setlength{\itemindent}{-12pt}
               }{\end{list}}

\title{Gambler's Ruin and the ICM}
\author{Persi Diaconis\footnote{Departments of Mathematics and Statistics, Stanford University.  Research supported by NSF grant DMS 1954042.} \and Stewart N. Ethier\footnote{Department of Mathematics, University of Utah. Research supported by Simons Foundation grant 429675.}}
\date{}

\begin{document}
\maketitle

\bigskip
``In the case where there are three players with limited fortunes, the various problems appear to be of
quite a different order of difficulty than in the case of two players.'' \quad Louis Bachelier (1912)
\bigskip

\begin{abstract}
Consider gambler's ruin with three players, 1, 2, and 3, having initial capitals $A$, $B$, and $C$ units.  At each round a pair of players is chosen (uniformly at random) and a fair coin flip is made resulting in the transfer of one unit between these two players.  Eventually, one of the players is eliminated and play continues with the remaining two. Let $\sigma\in S_3$ be the elimination order (e.g., $\sigma=132$ means player 1 is eliminated first and player 3 is eliminated second, leaving player 2 with $A+B+C$ units).

We seek approximations (and exact formulas) for the elimination order probabilities $P_{A,B,C}(\sigma)$.  Exact, as well as arbitrarily precise, computation of these probabilities is possible when $N:=A+B+C$ is not too large.  Linear interpolation can then give reasonable approximations for large $N$.  One frequently used approximation, the independent chip model (ICM), is shown to be inadequate.  A regression adjustment is proposed, which seems to give good approximations to the elimination order probabilities.\medskip

\end{abstract}

\section{Introduction}\label{sec:intro}

As motivation, first consider gambler's ruin with two players, 1 and 2, who initially have 1 and $N-1$ units.  At each round a fair coin flip is made resulting in the transfer of one unit from one player to the other.  Eventually, one of the players goes broke.  It is a classical result that
\begin{equation*}
P_{1,N-1}(\text{player 2 goes broke})=\frac1N.  
\end{equation*}

Consider next the game with three players having initial fortunes 1, 1, $N-2$.  At each round a pair of players is chosen (uniformly at random) and a fair coin flip is made resulting in the transfer of one unit between these two players.  What is
\begin{equation*}
P_{1,1,N-2}(\text{player 3 goes broke first})\,?
\end{equation*}

This basic problem has had little study.  A first thought is, ``Consider player 3 versus $\{1,2\}$.''  This is like gambler's ruin with two players.  Perhaps
\begin{equation*}
P_{1,1,N-2}(\text{player 3 goes broke first})\approx\frac{\text{constant}}{N}.
\end{equation*}
A well-studied scheme, the independent chip model (ICM), explained in Subsection~\ref{subsec:ICM} below, suggests
\begin{equation*}
P_{1,1,N-2}(\text{player 3 goes broke first})=\frac{2}{N(N-1)}.
\end{equation*}
We prove below that both of these are off. Indeed,
\begin{equation*}
P_{1,1,N-2}(\text{player 3 goes broke first})\approx\frac{\text{constant}}{N^3}.
\end{equation*}
It does not seem easy to give a simple heuristic for the $N^3$, and for $k\ge4$ players, the correct order of decay is open.

Let the initial capitals be $A$, $B$, and $C$ units, and put $N:=A+B+C$.  Let $\sigma\in S_3$ be the elimination order (e.g., $\sigma=132$ means player 1 is eliminated first and player 3 is eliminated second, leaving player 2 with $N$ units).  Useful approximations to $P_{A,B,C}(\sigma)$ are important in widely played versions of tournament poker; if, at the final table, three players remain, the first-, second-, and third-place finishers get fixed amounts $\alpha$, $\beta$, and $\gamma$, say, not depending on $A$, $B$, and $C$.  Clearly, $P_{A,B,C}(\sigma)$ is crucial in evaluating an equitable split of the prize pool $\alpha+\beta+\gamma$, should the players decide to ``settle.''  Such calculations are also required to evaluate the results of various actions throughout the game.

\begin{example}\label{wsop-example}
In the 2019 World Series of Poker Main Event, at the time the fourth-place finisher was eliminated, the three remaining players had chip counts as shown in Table~\ref{wsop-main} (WSOP, 2019a).

\begin{table}[htb]
\caption{\label{wsop-main}The final three in the 2019 World Series of Poker Main Event.}
\catcode`@=\active \def@{\hphantom{0}}
\tabcolsep=2mm 
\begin{center}
\begin{tabular}{lccc}
\hline\noalign{\smallskip}
@@@@player & chip count & big blinds & actual payoff \\
\noalign{\smallskip}\hline
\noalign{\smallskip}
Dario Sammartino & @$67{,}600{,}000$ & @33.8 & @$\$6{,}000{,}000$ \\
Alex Livingston & $120{,}400{,}000$ & @60.2 & @$\$4{,}000{,}000$ \\
Hossein Ensan   & $326{,}800{,}000$ & 163.4 & $\$10{,}000{,}000$ \\
\noalign{\smallskip}\hline
\noalign{\smallskip}
@@@@total          & $514{,}800{,}000$ & 257.4 & \\
\noalign{\smallskip}\hline
\end{tabular}
\end{center}
\end{table}

At this stage of the tournament, the standard unit bet\,---\,the big blind\,---\,was $2{,}000{,}000$ chips.  Initial capital, in big blinds, is shown for the three players in Table~\ref{wsop-main}, but to avoid fractions we multiply these numbers by 5 to get $A=169$, $B=301$, and $C=817$.  In the ensuing competition, the elimination order turned out to be 213, leaving Hossein Ensan with all $514{,}800{,}000$ chips and the \$10 million first-place prize.  The methods developed below (see Examples~\ref{wsop-interp3} and \ref{wsop-regression}) give the  chances shown in Table~\ref{wsop-chances} for the six possible elimination orders, \textit{assuming our random walk is a reasonable model for a no-limit Texas hold'em tournament}.  Thus, the second most likely elimination order is what actually occurred.

\begin{table}[htb]
\caption{\label{wsop-chances}The approximate probabilities of the six possible elimination orders in the scenario of Table~\ref{wsop-main}, assuming chip counts (in units of $400{,}000$ chips, or $1/5$ of the big blind) equal to $A=169$, $B=301$, and $C=817$.}
\catcode`@=\active \def@{\hphantom{0}}
\tabcolsep=2mm 
\begin{center}
\begin{tabular}{cccccccc}\hline
\noalign{\smallskip}
$\sigma$ & 123 & 132 & 213 & 231 & 312 & 321  \\
\noalign{\smallskip}
\hline
\noalign{\smallskip}
$P_{A,B,C}(\sigma)$  & 0.4196 & 0.2079 & 0.2152 & 0.1062 & 0.0260 & 0.0251 \\
\noalign{\smallskip}
\hline
\end{tabular}
\end{center}
\end{table}
\end{example}

Section \ref{sec:background} contains background on gambler's ruin and the independent chip model.  We review the connections with absorbing Markov chain theory.  This allows exact computation for $N$ up to at least 200.  Another approach, Jacobi iteration, allows virtually exact computation for $N$ up to at least 300.

We also observe that the $N=300$ data can be linearly interpolated to give reasonable approximations for arbitrary $N$.  One other method of approximation, based on a Monte Carlo technique, is described. 

Recent results for ``nice'' absorbing Markov chains (see Diaconis, Houston-Edwards, and Saloff-Coste, 2021) allow crude but useful approximations of $P_{A,B,C}(\sigma)$ uniformly.  The $\text{constant}/N^3$ result is proved as a consequence of that work.  

A new approximation approach is introduced in Section~\ref{sec:regression}.  The ratio
\begin{equation*}
P_{A,B,C}^{\text{GR}}(\sigma)/P_{A,B,C}^{\text{ICM}}(\sigma)
\end{equation*}
appears to be a smooth function of $A$, $B$, and $C$.  A sixth-degree polynomial regression is fit to this ratio and seen to give good approximations to $P_{A,B,C}^{\text{GR}}(\sigma)$.  In the sequel, superscripts GR (``gambler's ruin'') and ICM (``independent chip model'') will be used only when there is a chance of confusion.  (No superscript implicitly means GR.)

Section \ref{conjecture} gives some results for the gambler's ruin problem with $k\ge4$ players as well as a conjecture, namely the \textit{scaling conjecture}
\begin{equation*}
P_{A',B',C'}(\sigma)\doteq P_{A,B,C}(\sigma)\quad\text{whenever}\quad\frac{A'}{A}=\frac{B'}{B}=\frac{C'}{C},
\end{equation*}
where $\doteq$ denotes approximate equality.  (The symbol $\approx$ has a different meaning; see Theorem~\ref{D,HE,SC-thm} below.) 
An equivalent formulation,
\begin{equation}\label{scaling-conj2}
P_{nA,nB,nC}(\sigma)\doteq P_{A,B,C}(\sigma),\qquad n\ge2,
\end{equation}
may be preferable because it is closely related to the provable result that $\lim_{n\to\infty}P_{nA,nB,nC}(\sigma)$ exists;  indeed, the limit can be expressed in terms of standard two-dimensional Brownian motion.  A conjecture that is mathematically sharper than \eqref{scaling-conj2} appears in Subsection~\ref{subsec:scaling}.

Finally, Section~\ref{sec:summary} summarizes the various methods of evaluating and approximating the elimination order probabilities.  In all, six methods of approximation are studied, including a Brownian motion approximation along with the methods mentioned above.

\section*{Acknowledgments}

We thank Laurent Saloff-Coste, Laurent Miclo, Lexing Ying, Gene Kim, Sangchul Lee, Sourav Chatterjee, Guanyang Wang, Thomas Bruss, Pat Fitzsimmons, Bruce Hajek, Denis Denisov, Steve Stigler, Bernard Bru, Mason Malmuth, and Tom Ferguson for their help.  We are particularly indebted to Chris Ferguson for emphasizing the utility and mathematical depth of the problem.

\section{Background}\label{sec:background}

This section contains background on gambler's ruin\,---\,in two and higher dimensions (three or more players).  Exact computation of the Poisson kernel (harmonic measure) using absorbing Markov chains is taken up in Subsection~\ref{subsec:exact}, and arbitrarily precise computation by Jacobi iteration is the subject of Subsection~\ref{subsec:iteration}.  The use of barycentric coordinates to linearly interpolate these exact values is taken up in Subsection~\ref{subsec:interpolation}.  Another approaches to approximate computation, Monte Carlo methods, is described in Subsection~\ref{subsec:Monte}.

The asymptotics of the Poisson kernel are treated in Subsection~\ref{subsec:analytic}, which includes a proof of $P_{1,1,N-2}(\text{player 3 goes broke first})\approx\text{constant}/N^3$.  Finally, the ICM is introduced and its relation to the Plackett--Luce model is developed in Subsection~\ref{subsec:ICM}.

\subsection{Gambler's ruin}\label{subsec:ruin}

With two players, gambler's ruin is a classical topic, well developed in Feller (1968, Chap.\ XIV) and Ethier (2010, Chap.\ 7).  Important extensions to unfair coin flips and more-general step sizes are also well developed.  See Song and Song (2013) for a historical survey.

For $k=3$ players, the subject was first studied by Bachelier (1912).   The first post-Bachelier reference we have found is a formulation in terms of Brownian motion in a triangle due to Cover (1987).  This was solved by conformally mapping the triangle to a disk and using classical results for the Poisson kernel of the disk, by Hajek (1987) and later, independently, by Ferguson (1995).  Further results for $k=3$, including the poker connection, are in Kim (2005).

Martingale theory can be used to get information about the time to absorption.  For three players, let $T_1$ be the first time one of the three players is eliminated.  Bachelier (1912, \S204), Engel (1993), and Stirzaker (1994) proved
\begin{equation}\label{ET_1}
E(T_1)=\frac{3ABC}{A+B+C}.
\end{equation}
Thus, if $A=B=C=100$, then $E(T_1)=10{,}000$.  If $A=B=1$ and $C=298$, then $E(T_1)=2.98$.  Bruss, Louchard, and Turner (2003) and Stirzaker (2006) evaluated $\text{Var}(T_1)$.  Let $T_2$ be the first time two players are eliminated.  Bachelier (1912, \S 209), Engel (1993), and Stirzaker (1994) showed that
\begin{equation}\label{ET_2}
E(T_2)=AB+AC+BC.
\end{equation}
Thus, if $A=B=C=100$, then $E(T_2)=30{,}000$.  If $A=B=1$ and $C=298$, then $E(T_2)=597$.  Actually, Bachelier and Engel used first-order linear partial difference equations, while Stirzaker used martingales.  Bachelier's (1912, \S 209) proof of \eqref{ET_2} is very much worth reading.

A standard theorem (Bachelier 1912, \S 14) is 
\begin{equation}\label{P(3 wins)}
P(\text{player 3 wins all})=\frac{C}{A+B+C}.
\end{equation}
The results \eqref{ET_2} and \eqref{P(3 wins)} generalize to $k$ players.  There is a related development in the language of the ``Towers of Hanoi'' problem (Bruss, Louchard, and Turner, 2003; Ross, 2009).  None of this literature addresses the position at the first absorption time.

\subsection{Exact computation by Markov chain methods}\label{subsec:exact}

The gambler's ruin model is an example of an absorbing Markov chain in the state space
\begin{equation*}
\mathscr{X}:=\{(x_1,x_2,x_3)\in\textbf{Z}^3:x_1,x_2,x_3\ge0,\,x_1+x_2+x_3=N\}.
\end{equation*}
The first two coordinates determine things and the state space can be pictured (when $N=6$) as in Figure~\ref{statespace}.  The classical stars and bars argument shows that
\begin{equation*}
|\mathscr{X}|=\binom{N+2}{2},
\end{equation*}
and $\mathscr{X}$ has $\binom{N-1}{2}$ interior states, $3(N-1)$ nonabsorbing boundary states, and 3 absorbing states.  The Markov chain stopped at time $T_1$ is itself a Markov chain whose transition matrix 
can be written in block form as
\begin{equation*}
\bordermatrix{& \text{boundary} & \text{interior} \cr
\text{boundary} & \bm I & \bm 0 \cr
\text{interior} & \bm S & \bm Q \cr},
\end{equation*}
and elementary arguments yield the following theorem (Kemeny and Snell, 1976, Theorem 3.3.7).

\begin{figure}[htb]
\centering
\setlength{\unitlength}{1cm}
\begin{picture}(7,7.5)
\thinlines
\put(1,1){\line(1,0){6}}
\put(1,2){\line(1,0){5}}
\put(1,3){\line(1,0){4}}
\put(1,4){\line(1,0){3}}
\put(1,5){\line(1,0){2}}
\put(1,1){\line(0,1){6}}
\put(2,1){\line(0,1){5}}
\put(3,1){\line(0,1){4}}
\put(4,1){\line(0,1){3}}
\put(5,1){\line(0,1){2}}
\put(7,1){\line(-1,1){6}}
\put(6,1){\line(-1,1){5}}
\put(5,1){\line(-1,1){4}}
\put(4,1){\line(-1,1){3}}
\put(3,1){\line(-1,1){2}}
\put(1,1){\circle*{0.275}}
\put(2,1){\circle*{0.2}}
\put(3,1){\circle*{0.2}}
\put(4,1){\circle*{0.2}}
\put(5,1){\circle*{0.2}}
\put(6,1){\circle*{0.2}}
\put(7,1){\circle*{0.275}}
\put(1,2){\circle*{0.2}}
\put(2,2){\circle{0.2}}
\put(3,2){\circle{0.2}}
\put(4,2){\circle{0.2}}
\put(5,2){\circle{0.2}}
\put(6,2){\circle*{0.2}}
\put(1,3){\circle*{0.2}}
\put(2,3){\circle{0.2}}
\put(3,3){\circle{0.2}}
\put(4,3){\circle{0.2}}
\put(5,3){\circle*{0.2}}
\put(1,4){\circle*{0.2}}
\put(2,4){\circle{0.2}}
\put(3,4){\circle{0.2}}
\put(4,4){\circle*{0.2}}
\put(1,5){\circle*{0.2}}
\put(2,5){\circle{0.2}}
\put(3,5){\circle*{0.2}}
\put(1,6){\circle*{0.2}}
\put(2,6){\circle*{0.2}}
\put(1,7){\circle*{0.275}}
\end{picture}
\vglue-7mm
\caption{\label{statespace}When $N=6$, the state space $\mathscr{X}$ is represented by $28$ dots, of which $10$ are interior states (open dots), $15$ are nonabsorbing boundary states (solid dots), and $3$ are absorbing states (larger solid dots).  Line segments show possible transitions.  There are six from each interior state and two from each nonabsorbing boundary state.}
\end{figure}
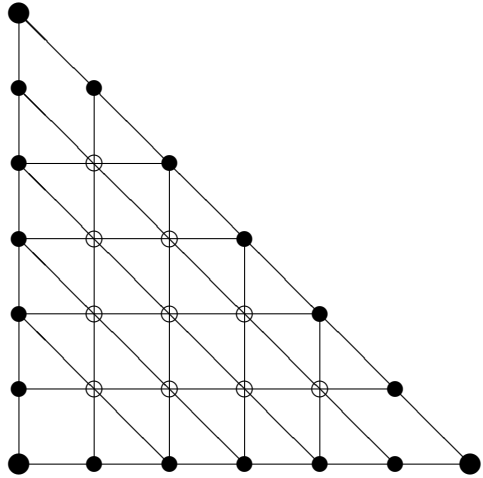

\begin{theorem}\label{exactPoisson}
For $\bm x\in\text{Int}(\mathscr{X})$ and $\bm y$ in the set of nonabsorbing boundary states of $\mathscr{X}$, define 
\begin{equation*}
P(\bm x,\bm y):=P_{\bm x}(\text{chain first reaches boundary at }\bm y),
\end{equation*}
so that $\bm P$ is an $\binom{N-1}{2}\times3(N-1)$ matrix.  Then
\begin{equation*}
\bm P=(\bm I-\bm Q)^{-1}\bm S.
\end{equation*}
\end{theorem}

The function $P(\bm x,\bm y)$ is called the \textit{Poisson kernel} or \textit{harmonic measure}.

\begin{example}
When $N=6$, $|\mathscr{X}|=\binom{6+2}{2}=28$, with the $\binom{6-1}{2}=10$ interior states ordered $114$, $123$, $132$, $141$, $213$, $222$, $231$, $312$, $321$, $411$, and the $3(6-1)=15$ nonabsorbing boundary states ordered $015$, $024$, $033$, $042$, $051$, $105$, $204$, $303$, $402$, $501$, $150$, $240$, $330$, $420$, $510$.  The Poisson kernel is given by Figure~\ref{Harmonic}.

\begin{figure}[htb]
\begin{center}
\includegraphics[height=1.9in]{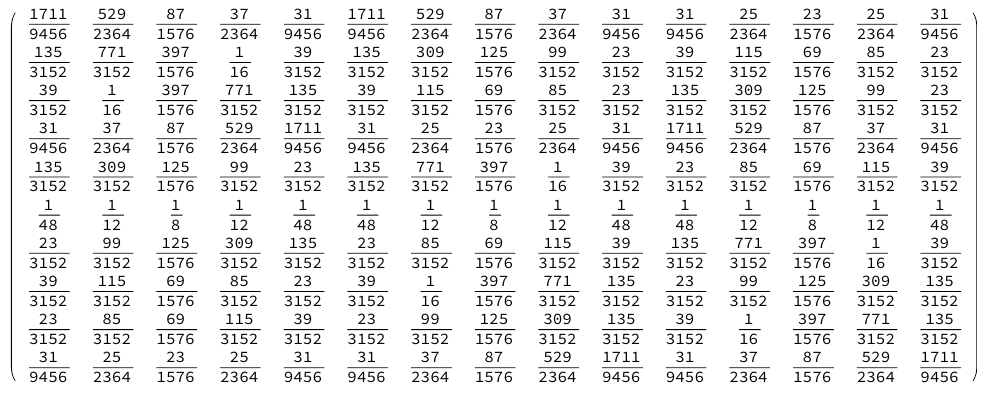}
\caption{\label{Harmonic}The Poisson kernel for $N=6$.  Rows are labeled by initial interior states ($114$, $123$, $132$, $141$, $213$, $222$, $231$, $312$, $321$, $411$), and columns by nonabsorbing boundary states ($015$, $024$, $033$, $042$, $051$, $105$, $204$, $303$, $402$, $501$, $150$, $240$, $330$, $420$, $510$).}
\end{center}
\end{figure}

From this we have the chance that the first absorption occurs at a given boundary point.  For the two remaining players, classical gambler's ruin gives the probability of the final outcome.  Summing over the appropriate part of the boundary gives the chances of the various elimination orders.  For $N=6$, these are given in Figure~\ref{EliminationOrder}.  Here the row ordering is as before, whereas the column ordering is $123$, $132$, $213$, $231$, $312$, $321$.

\begin{figure}[H]
\begin{center}
\includegraphics[height=1.9in]{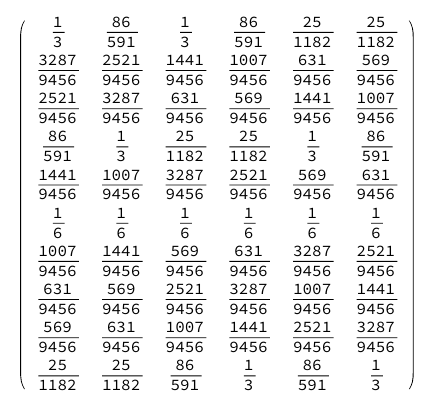}
\caption{\label{EliminationOrder}The probabilities of the six elimination orders for $N=6$.  Rows are labeled by initial interior states ($114$, $123$, $132$, $141$, $213$, $222$, $231$, $312$, $321$, $411$) and columns by elimination orders ($123$, $132$, $213$, $231$, $312$, $321$).}
\end{center}
\end{figure}
\end{example}

\textit{Mathematica} code, for arbitrary $N$, is provided in the supplementary materials (see Section~\ref{suppl}).  The only computationally difficult part of the program is inverting an $\binom{N-1}{2}\times\binom{N-1}{2}$ matrix.  When $N=200$ (the largest $N$ for which we have results), this matrix is $19{,}701\times 19{,}701$ and the program runtime (in double precision) was about 97 hours.  A faster alternative is described in Subsection~\ref{subsec:iteration}.

A very interesting paper by Swan and Bruss (2006) suggests that much larger  problems might be tackled. Their ideas apply to more general absorbing chains, but let us specialize to the three-player gambler's ruin. They partition the transient states into disjoint ``levels'' and observe that the transition matrix can be written as a block tridiagonal matrix (up to ``corner effects'') with considerably smaller blocks. Their second idea is to derive a ``folded'' chain on the even blocks. This has the same block tridiagonal form and so recursion can be used. Finally the absorption probabilities for the chain started in the odd blocks can be filled in. They do an order of magnitude calculation of the number of operations involved (along the lines of ``it takes order $n^3$ steps to invert an $n \times n$ matrix'') and conclude that the new algorithm would run a factor of $N^2$ steps faster than the straightforward matrix inversion we have used above.  The indexing is fairly sophisticated and we have not attempted to implement their fine ideas.

Using weighted directed multigraphs, David (2015) was able to reduce the number of transient states by about a factor of two.  His results, with $N$ as large as 192, are consistent with ours.  For application of this approach to four-player gambler's ruin, see Marfil and David (2020).

Gilliland, Levental, and Xiao (2007) found a way to avoid the inversion of large matrices in a one-dimensional gambler's ruin problem, but we have not been able to adapt their approach to the present setting.

These same techniques work for general absorbing Markov chains.  We have used them (supplementary materials, Section~\ref{suppl}) to compute the elimination order probabilities for $k=4$ players, requiring the inverse of an $\binom{N-1}{3}\times\binom{N-1}{3}$ matrix.  When $N=50$ (the largest $N$ for which we have results), this matrix is $18{,}424\times18{,}424$ and the program runtime (in single precision) was about 84.5 hours.  Here the walk takes place in a discrete 4-simplex.  Initial absorption is on one of the four triangular faces, and from there to final absorption one can apply the three-player results.

Our colleague Lexing Ying points out that the matrix $\bm I-\bm Q$ in Theorem~\ref{exactPoisson} is sparse (it has at most seven nonzero entries per row).  Sparse matrix inversion is a standard ``off the shelf'' tool in languages such as \textit{MATLAB}.  A useful textbook account is in Davis (2006).  Using these techniques, Ying was able to write code that generates results for $N$ as large as 3200 in about 2 minutes on a laptop computer.  He graciously agreed to determine the probabilities of the six elimination orders for the WSOP data of Example~\ref{wsop-example}, in which $N=1287$, and we

\subsection{Arbitrarily precise computation by Jacobi iteration}\label{subsec:iteration}

Fix an elimination order $\sigma\in S_3$ and total capital $N$.  Let $P_{A,B}$ be short for $P_{A,B,N-A-B}(\sigma)$.  Then, for $A,B\ge1$ with $A+B\le N-1$,
\begin{equation*}
P_{A,B}=\frac16(P_{A-1,B+1}+P_{A+1,B-1}+P_{A-1,B}+P_{A+1,B}+P_{A,B-1}+P_{A,B+1})
\end{equation*}
with boundary conditions determined by $\sigma$.  This may be used in two ways.  Start with any values for the $P_{A,B}$ agreeing with the boundary conditions, say all $P_{A,B}=\frac16$ except when $A=0$, $B=0$, or $A+B=N$.  Then repeatedly iterate this recurrence.  Again this may be done in two ways, either using (at stage $n$) $P^{n+1}$ in terms of $P^n$ or using updated values as they become available.  This method was used by Kim (2005) and seen to converge well for small values of $N$ (e.g., $N=16$).

A second approach harnesses a monotonicity property of the recurrence.  Let $P_{A,B}^*$ be the true gambler's ruin probabilities.  If $P_{A,B}^n\le P_{A,B}^*$ for all $A,B$, then $P_{A,B}^{n+1}\le P_{A,B}^*$ for all $A,B$.  Similarly for $P_{A,B}^n\ge P_{A,B}^*$.  Thus, starting the recurrence off with the correct boundary values and all other $P_{A,B}^{0,-}\equiv0$ and $P_{A,B}^{0,+}\equiv1$ gives
\begin{equation*}
P_{A,B}^{n,-}\le P_{A,B}^*\le P_{A,B}^{n,+}\text{ for all }A,B\text{ and }n.
\end{equation*}
When the lower and upper bounds are suitably close, this gives sharp control of $P_{A,B}^*$.  For a proof of convergence and further development, history, and references, see Ethier (2010, Theorem 7.2.4).

We adopt the latter approach, and we find that we can ensure the desired accuracy (18 significant digits) with $2N^2$ iterations.  \textit{Mathematica} code (for arbitrary $N$) is provided in the supplementary materials.  No matrix inversion is needed, so the program runs faster and uses much less memory than the one based on Markov chain methods.  When $N=200$, the program runtime (in double precision) was about 19 hours.  When $N=300$ (the largest $N$ for which we have results) it was about 98.5 hours.

The output of this program is a list of $P_{A,B,C}(123)$ for all $A,B,C\ge1$ with $A+B+C=N$.  If, for example, we want $P_{1,1,N-2}(321)$, we simply look up $P_{N-2,1,1}(123)$ instead.  Thus, there is no real loss of information in this condensed form of the output.

While this method allows for a larger $N$ in evaluating the three-player elimination order probabilities than the Markov chain method does ($N=300$ vs.\ $N=200$), the improvement is more significant in the four-player setting.  Here we generate the probabilities $P_{A,B,C,D}(1234)$ for all $A,B,C,D\ge1$ with $A+B+C+D=N$ and again find that $2N^2$ iterations suffice to ensure the desired accuracy (9 significant digits).  When $N=100$ (the largest $N$ for which we have results), the runtime was about 36 hours.  \textit{Mathematica} code is provided in the supplementary materials, but C++ code would run substantially faster.

\subsection{Linear interpolation from exact probabilities}\label{subsec:interpolation}

The virtually exact results for $N=300$ can be used to get useful approximations for other $N$.  Given positive integers $A$, $B$, and $C$, let $N:=A+B+C$ and
\begin{equation*}
A_0:=A\,\frac{300}{N},\quad B_0:=B\,\frac{300}{N},\quad C_0:=C\,\frac{300}{N}.
\end{equation*}
Typically, these are not integers.  Therefore, consider the four points 
\begin{align*}
\bm v_{00}&:=(\lfloor A_0\rfloor,\lfloor B_0\rfloor,300-\lfloor A_0\rfloor-\lfloor B_0\rfloor),\\
\bm v_{01}&:=(\lfloor A_0\rfloor,\lceil B_0\rceil,300-\lfloor A_0\rfloor-\lceil B_0\rceil),\\
\bm v_{10}&:=(\lceil A_0\rceil,\lfloor B_0\rfloor,300-\lceil A_0\rceil-\lfloor B_0\rfloor),\\
\bm v_{11}&:=(\lceil A_0\rceil,\lceil B_0\rceil,300-\lceil A_0\rceil-\lceil B_0\rceil),
\end{align*}
belonging to $\mathscr{X}$, and discard the one ($\bm v_{00}$ or $\bm v_{11}$) whose third coordinate is neither $\lfloor C_0\rfloor$ nor $\lceil C_0\rceil$.  The remaining three points, call them $(A_1,B_1,C_1)$, $(A_2,B_2,C_2)$, and $(A_3,B_3,C_3)$, form a triangle with $(A_0,B_0,C_0)$ belonging to its interior, and we can estimate $P_{A,B,C}(\sigma)$ by linear interpolation from the three values of $P_{A_i,B_i,C_i}(\sigma)$ ($i=1,2,3$).

The key idea is to represent $(A_0,B_0,C_0)$ in barycentric coordinates.  The relevant weights are
\begin{equation*}
\setlength{\arraycolsep}{2mm}
\binom{\lambda_1}{\lambda_2}:=\begin{pmatrix}A_1-A_3&A_2-A_3\\B_1-B_3&B_2-B_3\end{pmatrix}^{-1}\binom{A_0-A_3}{B_0-B_3}\quad\text{and}\quad
\lambda_3:=1-\lambda_1-\lambda_2,
\end{equation*}
so that
\begin{equation*}
(A_0,B_0,C_0)=\lambda_1(A_1,B_1,C_1)+\lambda_2(A_2,B_2,C_2)+\lambda_3(A_3,B_3,C_3),
\end{equation*}
and our interpolation estimate is then
\begin{equation*}
\bar P_{A,B,C}(\sigma):=\lambda_1 P_{A_1,B_1,C_1}(\sigma)+\lambda_2 P_{A_2,B_2,C_2}(\sigma)+\lambda_3 P_{A_3,B_3,C_3}(\sigma).
\end{equation*}

\begin{example}\label{wsop-interp3}
As described in Example~\ref{wsop-example}, the final three players in the 2019 WSOP Main Event had chip counts (in units of $400{,}000$ chips, or $1/5$ of the big blind) equal to $A=169$, $B=301$, and $C=817$.  Thus, $N=1287$ and $A$, $B$, and $C$, multiplied by $300/N$, are $A_0\doteq39.39$, $B_0\doteq70.16$, and $C_0\doteq190.44$.  It follows that $(A_1,B_1,C_1)=(39,70,191)$, $(A_2,B_2,C_2)=(39,71,190)$, and $(A_3,B_3,C_3)=(40,70,190)$.  The weights can then be evaluated as
\begin{equation*}
\lambda_1=\frac{190}{429},\quad\lambda_2=\frac{70}{429},\quad\lambda_3=\frac{13}{33},
\end{equation*}
and we can look up the probabilities $P_{A_i,B_i,C_i}(\sigma)$ for $i=1,2,3$ and each $\sigma$, with results shown in Table~\ref{table:wsop-interp3}.

\begin{table}[H]
\caption{\label{table:wsop-interp3}Linearly interpolating elimination order probabilities from $N=300$ data.  Here $A=169$, $B=301$, and $C=817$ from Example~\ref{wsop-example}.}
\tabcolsep=1.2mm 
\begin{center}
\begin{tabular}{ccccccc}
\hline
\noalign{\smallskip}
$\sigma$ & 123 & 132 & 213 & 231 & 312 & 321  \\
\noalign{\smallskip}\hline
\noalign{\smallskip}
$P_{39,70,191}(\sigma)$ & 0.422050 & 0.207786 & 0.214617 & 0.105295 & 0.025547 & 0.024705 \\
\noalign{\smallskip}
$P_{39,71,190}(\sigma)$ & 0.422204 & 0.210495 & 0.211129 & 0.104734 & 0.026172 & 0.025266 \\
\noalign{\smallskip}
$P_{40,70,190}(\sigma)$ & 0.415774 & 0.206898 & 0.217559 & 0.107757 & 0.026436 & 0.025576 \\
\noalign{\smallskip}\hline
\noalign{\smallskip}
$\bar P_{A,B,C}(\sigma)$ & 0.419603 & 0.207878 & 0.215207 & 0.106174 & 0.025999 & 0.025139 \\
\noalign{\smallskip}\hline
\end{tabular}
\end{center}
\end{table}

The scaling conjecture and observed smoothness of $P_{A,B,C}(\sigma)$ in $A$, $B$, and $C$ suggest that this will be a good approximation.  One way to assess the accuracy of the method is to use it to estimate probabilities that are already known; we have done so in several cases, and it appears that the interpolated probabilities are accurate to four or five decimal places.  See Example~\ref{wsop-regression} below for an alternative approach.  

Note that rounded proportions often do not sum precisely to 1.  See Diaconis and Freedman (1979).
\end{example}

\subsection{Monte Carlo methods}\label{subsec:Monte}

While the interpolation method of Subsection~\ref{subsec:interpolation} is our method of choice, this subsection records a further approximation method, Monte Carlo.  Guanyang Wang suggested a straightforward Monte Carlo procedure that approximates $P_{A,B,C}(\sigma)$ for a given $A,B,C\ge1$ and all $\sigma\in S_3$.  Simply run the Markov chain, starting at $(A,B,C)$, until it first reaches the boundary.  If $N:=A+B+C$ and the Markov chain first reaches the boundary at $(0,x,N-x)$, for example, then $\sigma=123$ and $\sigma=132$ are counted $(N-x)/N$ and $x/N$ times, by virtue of the two-player gambler's ruin formula.  Do this repeatedly, recording the proportion of times each $\sigma\in S_3$ occurs, and use these proportions as estimates.  A difficulty is that this procedure is rather slow.  For example, the expected number of steps for the Markov chain to first reach the boundary is given by \eqref{ET_1}, which is $96{,}876.4$ when using the WSOP data of Examples~\ref{wsop-example} and \ref{wsop-interp3} ($A=169$, $B=301$, and $C=817$).  

Wang suggested an optimization method to speed up the process.  Starting from state $(x_1,x_2,x_3)$, let $m=\min(x_1,x_2,x_3)$ and consolidate the next $m$ steps of the Markov chain into a single step by simulating $(n_1,n_2,n_3)\sim\text{multinomial}(m,\frac13,\frac13,\frac13)$, with $n_i$ representing the number of matchups in the next $m$ trials not involving player $i$, and $\zeta_i\sim\text{binomial}(n_i,\frac12)$ ($i=1,2,3$), with $\zeta_i$ representing the number of the $n_i$ matchups won by player $\text{mod}(i,3)+1$.

Wang has written \textit{R} code and shown that it works well for quite large $N$ (and also for $k=4$).  Starting with the just mentioned WSOP data ($A=169$, $B=301$, and $C=817$), the standard Monte Carlo procedure requires 2.7 seconds per sample path or 7.5 hours for sample size $10^4$, with the optimized procedure requiring only 0.0182 seconds per sample path or 3 minutes for sample size $10^4$ (148 times faster).  Example~\ref{wsop-regression} below compares simulation results (optimized, with sample size $10^6$) with other approximations.

\subsection{Analytic approximation}\label{subsec:analytic}

Some rather sophisticated analysis (John and inner uniform domains, Whitney covers, parabolic Harnack inequalities, Carlesson estimates) has been applied to get analytic approximations to the harmonic measure (Diaconis, Houston-Edwards, and Saloff-Coste, 2021).  The results apply to the $k$-player gambler's ruin problem, but we will content ourselves with the case $k=3$.  Code things up as in Figure~\ref{statespace} with two integer coordinates $x_1$, $x_2$ in the triangle $x_1,x_2\ge0$, $x_1+x_2\le N$.  This corresponds to $A=x_1$, $B=x_2$, and $C=N-x_1-x_2$.  By symmetry, it is enough to have approximations to 
\begin{equation*}
P(\bm x,(y,0)):=P_{\bm x}(\text{walk first reaches boundary at }(y,0))
\end{equation*}
with $\bm x=(x_1,x_2)$ in the interior of $\mathscr{X}$, satisfying $2x_1+x_2\le N$.  The boundary point $(y,0)$ has $0<y<N$.

\begin{theorem}[Diaconis, Houston-Edwards, Saloff-Coste, 2021]\label{D,HE,SC-thm}
For $x_1,x_2,y$ as above,
\begin{equation}\label{D,HE,SC-eq}
P(\bm x,(y,0))\approx\frac{x_1 x_2 (x_1+x_2)(N-x_1-x_2)(N-x_2)y^2(N-y)^2}{N^4(x_1+d)^2(x_2+d)^2(x_1+x_2+2d)^2}
\end{equation}
with $d$ being the graph distance from $\bm x$ to $(y,0)$.  Here $a_N\approx b_N$ means there exist positive $c$ and $c'$ (universal) such that 
\begin{equation*}
c\,a_N\le b_N\le c'a_N
\end{equation*}
for all $N$.  The constants implicit in \eqref{D,HE,SC-eq} are uniform for all $\bm x,y$.
\end{theorem}

Let us illustrate this result by proving the $1/N^3$ result claimed in Section~\ref{sec:intro}.

\begin{theorem}\label{O(N^{-3})}
\begin{equation*}
P_{1,1,N-2}(\text{\rm player 3 goes broke first})\approx \frac{1}{N^3}.
\end{equation*}
\end{theorem}

\begin{proof}
To get things into the notation of Theorem~\ref{D,HE,SC-thm}, take $x_1=1$, $x_2=N-2$.  Then, for $0<y<N$,
\begin{equation*}
P(\bm x,(y,0))=P(\text{player 2 goes broke first at which time player 1 has }y).
\end{equation*}
For any $y$, $d\approx N$ so the denominator in \eqref{D,HE,SC-eq} is $\approx N^{10}$.  The numerator is $\approx N^2 y^2(N-y)^2$.  Thus,
\begin{equation*}
P(\bm x,(y,0))\approx \frac {y^2(N-y)^2}{N^8}=\frac{1}{N^4}\bigg(\frac{y}{N}\bigg)^2\bigg(1-\frac{y}{N}\bigg)^2.
\end{equation*}
Summing in $y$ and reversing the roles of players 2 and 3,
\begin{align}\label{B(3,3)result}
&P_{1,1,N-2}(\text{player 3 goes broke first})\nonumber\\
&\quad{}\approx \frac{1}{N^3}\,\frac{1}{N}\sum_{y=1}^{N-1}\bigg(\frac{y}{N}\bigg)^2\bigg(1-\frac{y}{N}\bigg)^2\sim \frac{B(3,3)}{N^3},
\end{align}
where $B$ denotes the beta function.
\end{proof}

\begin{corollary}\label{2/N}
\begin{equation*}
P_{1,1,N-2}(\text{\rm player 3 goes broke second})\sim\frac{2}{N}.
\end{equation*}
\end{corollary}

\begin{proof}
The desired probability is
\begin{align*}
&{}1-P_{1,1,N-2}(\text{\rm player 3 wins all})-P_{1,1,N-2}(\text{\rm player 3 goes broke first})\\
&\qquad\quad{}=1-\frac{N-2}{N}-O(1/N^3)=\frac2N-O(1/N^3)
\end{align*}
by \eqref{P(3 wins)} and Theorem~\ref{O(N^{-3})}, and the result follows.
\end{proof}

\begin{remarks}
\begin{enumerate}[(1)]

\item The constant $B(3,3)=1/30$ in \eqref{B(3,3)result} is meaningless because of all the cruder approximations being used.  Now
\begin{equation*}
P_{1,1,N-2}(\text{player 3 goes broke first})=P_{1,1,N-2}(312)+P_{1,1,N-2}(321),
\end{equation*}
and because of symmetry,
\begin{equation*}
P_{1,1,N-2}(312)=P_{1,1,N-2}(321),
\end{equation*}
so Theorem~\ref{O(N^{-3})} implies
\begin{equation*}
P_{1,1,N-2}(312)=P_{1,1,N-2}(321)\approx\frac{1}{N^3}.
\end{equation*}

\item A similar calculation shows, for $1\le i<N/2$,
\begin{equation*}
P_{i,i,N-2i}(\text{player 3 goes broke first})\approx\frac{i^3}{N^3},
\end{equation*}
uniformly in $i$.  This is consistent with the scaling conjecture of Section~\ref{conjecture}.

\item The asymptotics above may be supplemented by the exact computing of Subsections~\ref{subsec:exact} and \ref{subsec:iteration}.  Table~\ref{asymp} gives $P_{1,1,N-2}(321)$ for $N=50$, 100, 150, 200, 250, 300 as well as these values multiplied by $N^3$.

\begin{table}[htb]
\caption{\label{asymp}The exact values of $P_{1,1,N-2}(321)$, rounded to 15 significant digits, suggesting that this quantity is asymptotic to $c/N^3$ for $c\doteq4.5597945$.}
\catcode`@=\active \def@{\hphantom{0}}
\tabcolsep=2mm 
\begin{center}
\begin{tabular}{clc}
\hline
\noalign{\smallskip}
$N$ & @@@$P_{1,1,N-2}(321)$  & $N^3 P_{1,1,N-2}(321)$  \\
\noalign{\smallskip}\hline
\noalign{\smallskip}
@50  & 0.0000364783779008280  & 4.55979723760 \\
100 & 0.00000455979467170448 & 4.55979467170 \\
150 & 0.00000135105023226911 & 4.55979453391 \\
200 & 0.000000569974313837992& 4.55979451070 \\
250 & 0.000000291826848279112& 4.55979450436 \\
300 & 0.000000168881277854908& 4.55979450208 \\
\noalign{\smallskip}\hline
\noalign{\smallskip}
\end{tabular}
\end{center}
\end{table}

\item In unpublished work, Sangchul Lee has used Ferguson's (1995) Brownian motion approximation to the discrete gambler's ruin problem to derive an analytical closed form expression for the constant $c\doteq4.5597945$ in Table~\ref{asymp}.  He shows
\begin{equation*}
c=\frac{\sqrt{\pi}}{3\sqrt3}\bigg(\frac{\Gamma(1/3)}{\Gamma(5/6)}\bigg)^3\doteq4.55979449996,
\end{equation*}
in remarkable agreement to the numbers in Table~\ref{asymp}.  The validity of the Brownian motion approximation has not been rigorously established to this degree.  See Denisov and Wachtel (2015).

\item Theorem~\ref{D,HE,SC-thm} allows proof of similar asymptotics for other values of $A$, $B$, and $C$.  For example, we have proved the following:
\begin{itemize}
\item For fixed $A,B\ge1$ and $C_N:=N-A-B$,
\begin{equation*}
P_{A,B,C_N}(321)\approx P_{A,B,C_N}(312)\approx\frac{1}{N^3}.
\end{equation*}
\item For $A=1$, $B_N=\lfloor\sqrt{N}\rfloor$, and $C_N:=N-1-B_N$,
\begin{equation*}
P_{1,B_N,C_N}(\text{player 3 goes broke first})\approx\frac{1}{N^2}.
\end{equation*}
\end{itemize}
Exact computations suggest that in the first case $N^3 P_{A,B,C_N}(321)$ and in the second case $N^2 P_{A,B_N,C_N}(\text{player 3 goes broke first})$ rapidly approach limits.

\item Similarly, $P_{1,B_N,C_N}(321)\approx P_{1,B_N,C_N}(312)\approx1/N^2$. This is a bit surprising.  Of course, the event that the player with the big stack is eliminated first is a rare event but then the advantage that player 2 had over player 1 disappears. Indeed, numerical computations show that, for this case, given player 3 is eliminated first, the conditional gambler's ruin probability that 1 is eliminated second is 1/2 to remarkable approximation.  For example, $P_{1,14,185}(321)\doteq0.000059822$ and $P_{1,14,185}(312)\doteq0.000059872$.

\item The results of Diaconis, Houston-Edwards, and Saloff-Coste (2021) were not intended to give good numerics.  We hope that comparing them to data will allow better choices of omitted constants as in item (3) above.  The following results are examples.

\item For any $N$,
\begin{align*}
P_{1,1,N-2}(123)=P_{1,1,N-2}(213)&=\frac12 P_{1,1,N-2}(\text{player 3 wins all})\\
&=\frac12\,\frac{N-2}{N}=\frac12\bigg(1-\frac2N\bigg).
\end{align*}
When $N=200$, the right side is 0.495.  Using Theorem~\ref{D,HE,SC-thm},
\begin{equation*}
P_{1,1,N-2}(123)\approx\frac12\sum_{y=1}^{N-1}\frac{(1-y/N)^3}{y^4}=\frac{\zeta(4)}{2}(1+o(1))\doteq 0.5412.
\end{equation*}
Similarly, if $1\le i<N/2$,
\begin{equation*}
P_{i,i,N-2i}(123)=P_{i,i,N-2i}(213)=\frac12\,\frac{N-2i}{N}=\frac12\bigg(1-\frac{2i}{N}\bigg),
\end{equation*}
confirming the scaling conjecture in this case.  So perhaps not all hope is lost for using Theorem~\ref{D,HE,SC-thm}.

\item Similarly, taking $x_1=x_2=1$,
\begin{align*}
P_{1,1,N-2}(132)&=P_{1,1,N-2}(231)\approx\frac12\sum_{y=1}^{N-1}\frac{(1-y/N)^2}{y^4}\,\frac{y}{N}\\
&=\frac{1}{2N}\sum_{y=1}^{N-1}\frac{(1-y/N)^2}{y^3}\sim\frac{\zeta(3)}{2N}.
\end{align*}
By Corollary~\ref{2/N}, these probabilities are asymptotic to $1/N$, so this estimate is off by a factor of $\zeta(3)/2\doteq0.6010$.
\end{enumerate}
\end{remarks}

\subsection{The independent chip model (ICM)}\label{subsec:ICM}

There are a variety of reasons for wanting to compute the chances of the various elimination orders.  The most classical one, ``The Problem of Points,'' has to do with splitting the capital in a $k$-player game when the game must be called off early.  This is one of the problems that got Fermat and Pascal in correspondence\,---\,the start of modern probability theory.  In tournament poker, we have seen three players decide to ``settle,'' dividing the final prize money in proportion to their current chip totals.  (As we will see, this is not the right way to do it.)  Of course, calculating expectations for various decisions (mentioned earlier) is a key application.

The independent chip model (ICM), a popular scheme, originated in a 1986 article by Mason Malmuth in \textit{Poker Player Newspaper}, which was reprinted in Malmuth (1987, 2004).  Although the name came later, the concept was used to argue that rebuying in a percentage-payback poker tournament is mathematically correct, contrary to conventional wisdom at the time.  Other implications of the ICM for poker tournaments were discussed by Gilbert (2009).  See Aguilar (2016) for its use in ``chopping'' the prize pool in poker tournaments, using a poker ICM calculator (ICMizer, 2020).

ICM builds on a solid foundation:  In the two-player gambler's ruin problem for fair coin-tossing, if player 1 starts with $A$ and player 2 starts with $B$, the chance that player 1 (respectively, player 2) wins all is $A/(A+B)$ (resp., $B/(A+B)$).  Now a heuristic step: Consider three players with initial capitals $A$, $B$, and $C$.  The chance that a given player wins all is (rigorously) proportional to his initial capital (so the chance that player 1 wins all is $A/N$, where $N:=A+B+C$).  The ICM calculation conditions on this, uses the relative initial capitals of the two nonwinners to calculate the chance of being second eliminated, \textit{and then multiplies}.  This results in the chances shown in Table~\ref{ICM-table} assigned to the six elimination orders.

\begin{table}[htb]
\caption{\label{ICM-table}The ICM with three players.}
\tabcolsep=2mm 
\begin{center}
\begin{tabular}{ccccccc}
\hline
\noalign{\smallskip}
$\sigma$ & 123 & 132 & 213 & 231 & 312 & 321  \\
\noalign{\smallskip}\hline
\noalign{\smallskip}
$P_{A,B,C}^{\text{ICM}}(\sigma)$ & $\frac{C}{N}\,\frac{B}{A+B}$ & $\frac{B}{N}\,\frac{C}{A+C}$ & $\frac{C}{N}\,\frac{A}{A+B}$ & $\frac{A}{N}\,\frac{C}{B+C}$ & $\frac{B}{N}\,\frac{A}{A+C}$ & $\frac{A}{N}\,\frac{B}{B+C}$ \\
\noalign{\smallskip}\hline
\noalign{\smallskip}
\end{tabular}
\end{center}
\end{table}

The probabilities for $k\ge4$ players are determined similarly.

We can now be more explicit about how the prize pool is chopped when the last three players decide to settle.  First, apportioning it in proportion to current chip totals does not take prize money into account and is unsupportable.  For example, in Example~\ref{wsop-example}, the chip leader would get about \$12.696\,M, more than he would get by finishing first, and the player in third place would get about \$2.626\,M, less than he would get by finishing third.  So let $P(\sigma)$ be the probability of elimination order $\sigma$.  Let $\alpha$. $\beta$, and $\gamma$ be the payouts for first, second, and third place.  Then the expression for the amounts apportioned to players 1, 2, and 3 should be
\begin{align*}
&(\gamma,\beta,\alpha)P(123)+(\gamma,\alpha,\beta)P(132)+(\beta,\gamma,\alpha)P(213)\\
&\qquad{}+(\alpha,\gamma,\beta)P(231)+(\beta,\alpha,\gamma)P(312)+(\alpha,\beta,\gamma)P(321).
\end{align*}
It is standard practice to use $P_{A,B,C}^{\text{ICM}}(\sigma)$ in place of $P(\sigma)$.  Alternatively, one could use $\bar P_{A,B,C}^{\text{GR}}(\sigma)$.  Of course, both are approximations.

\begin{example}
Let us return to Example~\ref{wsop-example} with $A=169$, $B=301$, and $C=817$; also $\alpha=\$10$\,M, $\beta=\$6$\,M, and $\gamma=\$4$\,M.  Using the ICM probabilities from Table~\ref{ICM-table}, we get $(\$5.325\,\text{M}, \$6.287\,\text{M}, \$8.388\,\text{M})$, which is consistent with ICMizer (2020), whereas using the interpolated gambler's ruin probabilities from Example~\ref{wsop-interp3}, we get $(\$5.270\,\text{M}, \$6.293\,\text{M}, \$8.437\,\text{M})$.  The player with third-largest chip total gets about \$54\,K more from an ICM chop than from a GR chop.

\end{example}

\begin{remarks}
\begin{enumerate}[(1)]

\item \textit{ICM is different from gambler's ruin}.  Consider $N=6$ and initial capital $A=1$, $B=2$, and $C=3$.  What is the chance the elimination order is 321?  Using the exact calculation in Figure~\ref{EliminationOrder} and the ICM formula yields
\begin{equation*}
P_{1,2,3}^{\text{GR}}(321)=\frac{569}{9456}\doteq0.06017, \qquad P_{1,2,3}^{\text{ICM}}(321)=\frac{1}{6}\,\frac{2}{5}\doteq0.06667.
\end{equation*}

\item \textit{The results can be of different orders of magnitude}.  With starting capitals $1,1,N-2$, 
\begin{equation*}
P_{1,1,N-2}^{\text{GR}}(321)\approx\frac{1}{N^3},\qquad P_{1,1,N-2}^{\text{ICM}}(321)=\frac{1}{N}\,\frac{1}{N-1}\sim\frac{1}{N^2}.
\end{equation*}

\item \textit{Sometimes they agree}.  With starting capitals $i,i,N-2i$, where $1\le i<N/2$,
\begin{equation*}
P_{i,i,N-2i}^{\text{GR}}(123)=P_{i,i,N-2i}^{\text{ICM}}(123)=\frac{N-2i}{N}\,\frac12=\frac12\bigg(1-\frac{2i}{N}\bigg).
\end{equation*}

\item \textit{They are often quite different}.  In the next section we calculate of the ratios
\begin{equation*}
P_{A,B,C}^{\text{GR}}(\sigma)/P_{A,B,C}^{\text{ICM}}(\sigma)
\end{equation*}
for all $A,B,C\ge1$ with $A+B+C=300$ and all $\sigma\in S_3$.  The ratios vary considerably, ranging from about 0.015 to about 1.15.

\item \textit{But} poker is a complicated game, particularly no limit where the bets can be arbitrary.  The gambler's ruin model is based on single-unit bets.  Why is this relevant?  Some variants of the $\pm1$ transfer have been studied.

\begin{itemize}
\item \textit{All in}:  After two players out of the remaining $k$ are chosen, if they have $A$ and $B$ respectively, the bet size is $\min(A,B)$.  The player with the smaller chip count is eliminated or doubles up.

\item \textit{Occasionally all in}:  This is a compromise between unit bets and all-in bets.  After two players out of the remaining $k$ are chosen, if they have $A$ and $B$ respectively, the bet size is chosen uniformly at random from $\{1,2,\ldots,\min(A,B)\}$.

\item \textit{Compulsive gambler} (Aldous, Lanoue, and Salez, 2015):  After two players out of the remaining $k$ are chosen, one gets the other's money with probabilities given by the two-player gambler's ruin formula.  That is, if the respective amounts are $A$ and $B$, the player with $A$ wins (and then has $A+B$) with probability $A/(A+B)$, or loses (and is eliminated) with probability $B/(A+B)$.
\end{itemize}

A fascinating effort at finding an optimal strategy for $k$-player gambler's ruin with all-in betting is in Ganzfried and Sandholm (2008). Interestingly, they use ICM as a starting evaluation of the value function and then sharpen this using fictitious play and value iteration.

These variants will (almost surely) result in different elimination order probabilities.  The ICM assignment is different yet again.  Thus, there are many distinct models.  It would be worthwhile to look at some of the available data for tournament poker and compare.  We wouldn't be surprised if all these models are inadequate.
\end{enumerate}
\end{remarks}\medskip

The next section salvages something from these differences, using the ratios and regression to give a useful approximation to the gambler's ruin probabilities.

To finish this section, let us note that ICM is well studied as the Plackett--Luce model.  This is a model allowing non-uniform distributions on $S_k$, the set of permutations of $k$ distinct items, labeled $1,2,\ldots,k$.  Each item $i$ is assigned a weight $w_i>0$ with $w_1+w_2+\cdots+w_k=w$.  Now imagine these weights placed in an urn and the weights removed sequentially, each time with probability proportional to its size among the remaining weights.  Thus,
\begin{equation*}
P(\sigma):=\frac{w_{\sigma(1)}}{w}\,\frac{w_{\sigma(2)}}{w-w_{\sigma(1)}}\,\frac{w_{\sigma(3)}}{w-w_{\sigma(1)}-w_{\sigma(2)}}\cdots.
\end{equation*}

The model was introduced in perception psychology by R. Duncan Luce (1959, 1977).  It has a variety of derivations: via the elimination by aspect axiom; as the distribution of the order statistics of independent exponential variables (the $i$th having mean $w_i$); and as the stationary distribution of the Tsetlin library.  See Diaconis (1988, pp.~174--175) for further references.

Later reinventions of the model were published by Harville (1973) and Plackett (1975), both of whom applied it to horse racing, and it seems to have a life of its own for this application (Stern, 2008).  There is good available code for fitting this model to data (Turner et al., 2017) and many applications.  Although the model is referred to in the literature as the Plackett--Luce model, perhaps Luce--Harville--Plackett--Malmuth would be chronologically more correct.

Finally, we note that enumerative combinatorics for the Plackett--Luce model can be interesting and challenging;  What is the approximate distribution of the number of fixed points or cycles, and how does it depend on the weights?

\section{ICM and regression for gambler's ruin}\label{sec:regression}

Here we show how to use the easy-to-compute ICM probabilities $P_{A,B,C}^{\text{ICM}}(\sigma)$ to get surprisingly good approximations to the gambler's ruin probabilities $P_{A,B.C}^{\text{GR}}(\sigma)$.  Throughout, we work with $k=3$ players, fair coin flips, and $\pm1$ transfers at each stage.

We can base the analysis on the $N=300$ data, which gives $P_{A,B,C}(123)$ (in double precision) for all $A,B,C\ge1$ with $A+B+C=300$.  There are $\binom{300-1}{2}=44{,}551$ such points.  Notice that, for $\sigma=\sigma(1)\,\sigma(2)\,\sigma(3)$,
\begin{equation*}
P_{A_1,A_2,A_3}(\sigma)=P_{A_{\sigma(1)},A_{\sigma(2)},A_{\sigma(3)}}(123),
\end{equation*}
so there is no loss of information by restricting to $\sigma=123$.  

As efficient as this data set is, there is still some redundancy in the data, as has already been alluded to in \eqref{P(3 wins)} and elsewhere, namely
\begin{align}\label{mart-identities}
P_{A,B,C}(123)+P_{A,B,C}(213)&=\frac{C}{A+B+C},\nonumber\\
P_{A,B,C}(132)+P_{A,B,C}(312)&=\frac{B}{A+B+C},\\
P_{A,B,C}(231)+P_{A,B,C}(321)&=\frac{A}{A+B+C},\nonumber
\end{align}
a consequence of the optional stopping theorem.  The result is that it suffices to consider only one of the two probabilities in each row of \eqref{mart-identities}.  Incidentally, the equations in \eqref{mart-identities} hold trivially with superscript ICM. 

We begin by evaluating, for $N=300$ and $\sigma=123$, the ratios
\begin{equation}\label{ratio123}
R_\sigma(A,B,C):=P_{A,B,C}^{\text{GR}}(\sigma)/P_{A,B,C}^{\text{ICM}}(\sigma)
\end{equation}
for all $A,B,C\ge1$ with $A+B+C=N$.  As already noted, there are $44{,}551$ such ratios and all of them belong to $(0.015,1.15)$.  The function $R_{123}$ is plotted in Figure~\ref{RatioPlot}.

\begin{figure}[htb]
\begin{center}
\includegraphics[width=2.25in]{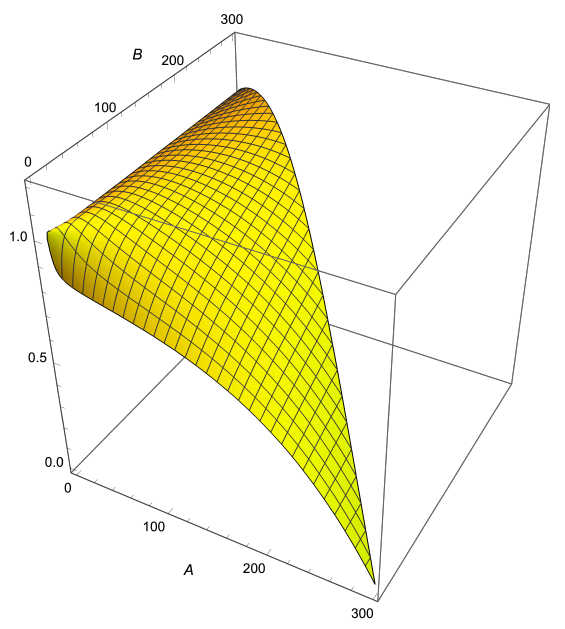}
\includegraphics[width=2.25in]{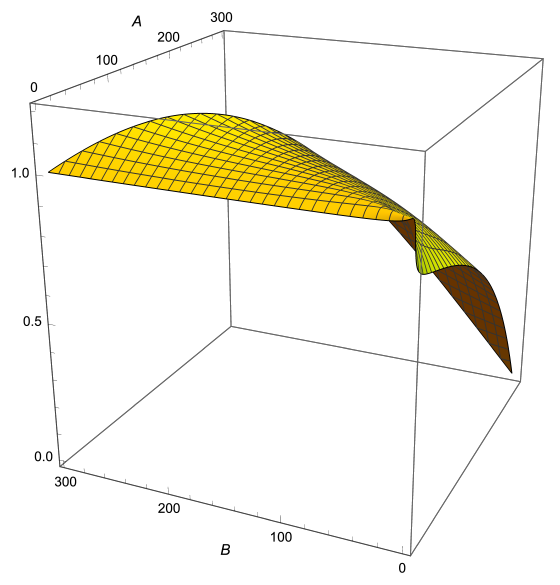}
\caption{\label{RatioPlot}A plot of $R_{123}$ defined in \eqref{ratio123} as a function of $(A,B)$ when $N=300$.  (The domain of $R_{123}$ is restricted to $A,B\ge1$ with $A+B\le N-1$.).  The second figure is a rotation of the first that reveals the singularity near $(1,1)$.}
\end{center}
\end{figure}

Notice that $R_{123}$ appears smooth as a function of $(A,B)$ ($C=N-A-B$), except for a singularity near $(1,1)$.  We can mitigate the effect of the singularity by  considering $R_{\sigma}$ over $1\le A\le B\le C$ with $A+B+C=N$ for $\sigma=213$, 312, and 321.  (The number of such triples $(A,B,C)$ is $N^2/12$ if $N$ is divisible by 6, hence 7500 if $N=300$.) In each case we fit a sextic polynomial in
\begin{equation*}
x:=\frac{A}{N}\quad\text{and}\quad y:=\frac{B}{N}
\end{equation*}
to the function $R_{\sigma}$.  A quadratic approximation does not give very good results, while a quartic approximation is quite good, and a sextic is even better.  At the same time, the higher the degree, the closer the design matrix is to being less than full rank.  An octic approximation results in some disturbingly large estimated regression coefficients, so we have settled on a sextic polynomial approximation.  Thus, we want to approximate $R_{\sigma}$ by the polynomial with 28 terms
\begin{equation*}
p_{\sigma}(x,y):=\sum_{i,j\ge0,\,i+j\le6}\beta_{ij}x^i y^j.
\end{equation*}

Let $\bm Y$ be the column vector of values of $R_{321}$ (with $N=300$), indexed by the vectors $(A,B,C)$ (with $1\le A\le B\le C$ and $A+B+C=N$) ordered lexicographically, let $\bm X$ be the matrix whose rows are indexed as the entries of $\bm Y$, and with row $(A,B,C)$ containing $1$, $x$, $y$, $x^2$, $xy$, $y^2$, $x^3$, $x^2 y$, \dots, $y^6$, where $x=A/N$ and $y=B/N$.  Note that $\bm Y$ has length 7500 and $\bm X$ is 7500 by 28.  To quantify the claim that $\bm X'\bm X$ becomes closer to being singular as the degree of the approximating polynomial increases, we note that, with $N=300$, $\det(\bm X'\bm X)$ is $1.68\times10^6$ for quadratic approximation, $3.90\times10^{-29}$ for quartic, $1.14\times10^{-136}$ for sextic, and $1.10\times10^{-415}$ for octic.

The estimated regression coefficients are
\begin{equation*}
\hat{\bm\beta}=(\bm X'\bm X)^{-1}\bm X'\bm Y,
\end{equation*}
and the values of the fitted polynomial $\hat p_{321}$ are the entries of $\bm X\hat{\bm\beta}$.  Table~\ref{coeffs} lists the estimated regression coefficients, and the error sum of squares is $7.86\times10^{-9}$ for $\sigma=321$, $8.95\times10^{-9}$ for $\sigma=312$, and 0.0177 for $\sigma=213$.  Additional detail is given in the supplementary materials (Section~\ref{suppl}).

This gives the approximation 
\begin{equation}\label{phat}
\hat P_{A,B,C}^{\text{GR}}(321):=P_{A,B,C}^{\text{ICM}}(321)\;\hat p_{321}\bigg(\frac{A}{N},\frac{B}{N}\bigg)=
\frac{A}{N}\,\frac{B}{B+C}\;\hat p_{321}\bigg(\frac{A}{N},\frac{B}{N}\bigg),
\end{equation}
and the cases $\sigma=312$ and $\sigma=213$ are treated in the same way.  The derivation assumed $N=300$ throughout.  We did the same computation for $N=200$, and the estimated regression coefficients did not change much, indicating stability.   We expect the approximation to be reasonable for other (perhaps much larger) values of $N$.  That is, for general $N$ use the approximation \eqref{phat} in which the function $\hat p_{321}$ is determined by the coefficients in Table~\ref{coeffs} computed from the $N=300$ data.  We investigate this in two examples below.

\begin{example}\label{regression-accuracy}
Table~\ref{exact-regression} compares exact values of $P_{A,B,C}(\sigma)$ with its interpolation approximation $\bar P_{A,B,C}(\sigma)$ and its regression-corrected ICM $\hat P_{A,B,C}(\sigma)$.  In the two examples, which are representative, we find that, for $\sigma=321$ and $\sigma=312$ (and their ``complements'' $\sigma=231$ and $\sigma=132$), the regression approximation is often accurate to six significant digits (except near the boundary of $\mathscr{X}$).  But with $\sigma=213$ (and $\sigma=123$) the regression approximation is not as good, perhaps only three or four significant digits.  In the latter case, we see from Table~\ref{coeffs} that the estimated regression coefficients are substantially larger, which is indicative of a poorer fit.  On the other hand, the interpolation approximation is typically accurate to four or five decimal places.

\begin{table}[H]
\caption{\label{coeffs}The estimated regression coefficients in fitting a sextic polynomial in $x:=A/N$ and $y:=B/N$ to $P_{A,B,C}^{\text{GR}}(\sigma)/P_{A,B,C}^{\text{ICM}}(\sigma)$, when $1\le A\le B\le C$ and $A+B+C=N$.  Here $N=300$.}
\tabcolsep=2mm 
\catcode`@=\active \def@{\hphantom{0}}
\catcode`#=\active \def#{\hphantom{$-$}}
\begin{small}
\begin{center}
\begin{tabular}{cccc}
\hline
\noalign{\smallskip}
& $\sigma=321$ & $\sigma=312$ & $\sigma=213$ \\
\noalign{\smallskip}\hline
\noalign{\smallskip}
$\hat\beta_{00}$ & #0.00000716459 & $-0.00000434510$ & #@@@0.951694   \\
$\hat\beta_{10}$ & #2.27836@@@      & #@2.27978@@      & #@@@7.07267@   \\
$\hat\beta_{01}$ & #2.28007@@@      & #@2.28028@@      & @@@$-3.74069$@ \\
$\hat\beta_{20}$ & $-2.25895$@@@    & #@0.0295644      & @@$-85.5467$@@ \\
$\hat\beta_{11}$ & $-2.24587$@@@    & @$-2.30603$@@    & @@$-23.1325$@@ \\
$\hat\beta_{02}$ & $-0.00740617$    & @$-2.28407$@@    & #@@37.7612@@   \\
$\hat\beta_{30}$ & #0.0793010@      & @$-0.618044$@    & #@336.603@@@   \\
$\hat\beta_{21}$ & $-0.581954$@@    & #@0.285484@      & #@646.261@@@.  \\
$\hat\beta_{12}$ & $-0.191389$@@    & #@0.241672@      & @$-111.109$@@@ \\
$\hat\beta_{03}$ & #0.0630814@      & #@0.0171583      & @$-197.597$@@@ \\
$\hat\beta_{40}$ & #1.48069@@@      & #@0.533492@      & @$-557.212$@@@ \\
$\hat\beta_{31}$ & #7.41061@@@      & @$-3.71135$@@    & $-2327.47$@@@@ \\
$\hat\beta_{22}$ & #3.41508@@@      & @$-4.35883$@@    & $-1723.69$@@@@ \\
$\hat\beta_{13}$ & $-5.89118$@@@    & #@5.79382@@      & #@874.263@@@   \\
$\hat\beta_{04}$ & $-2.67189$@@@    & #@2.49997@@      & #@540.005@@@   \\
$\hat\beta_{50}$ & $-0.556140$@@    & @$-4.34151$@@    & #@401.498@@@   \\
$\hat\beta_{41}$ & $-3.73983$@@@    & $-14.9318$@@@    & #2829.27@@@@   \\
$\hat\beta_{32}$ & $-5.73258$@@@    & #14.4755@@@      & #5167.69@@@@   \\
$\hat\beta_{23}$ & #0.0461851@      & #13.5284@@@      & #1655.90@@@@   \\
$\hat\beta_{14}$ & #1.48366@@@      & @$-7.28960$@@    & $-1925.03$@@@@ \\
$\hat\beta_{05}$ & $-0.274686$@@    & @$-2.71321$@@    & @$-746.323$@@@ \\
$\hat\beta_{60}$ & $-0.0378116$@    & #@2.48372@@      & @$-100.681$@@@ \\
$\hat\beta_{51}$ & $-1.05892$@@@    & #21.5551@@@      & $-1101.29$@@@@ \\
$\hat\beta_{42}$ & $-4.84129$@@@    & #@6.48502@@      & $-3408.10$@@@@ \\
$\hat\beta_{33}$ & $-5.77799$@@@    & $-35.4065$@@@    & $-3643.90$@@@@  \\
$\hat\beta_{24}$ & #1.08693@@@      & $-12.3409$@@@    & @$-213.799$@@@ \\
$\hat\beta_{15}$ & #5.19094@@@      & #@5.44466@@      & #1420.66@@@@   \\
$\hat\beta_{06}$ & #1.88817@@@      & #@1.07997@@      & #@410.303@@@   \\
\noalign{\smallskip}\hline
\noalign{\smallskip}
\end{tabular}
\end{center}
\end{small}
\end{table}

\begin{example}[Example \ref{wsop-interp3} continued]\label{wsop-regression}
Recall that, in Example~\ref{wsop-interp3}, we estimated $P_{A,B,C}(\sigma)$ when $A=169$, $B=301$, and $C=817$.  We did so using linear interpolation based on the $N=300$ data.  Results are restated in Table~\ref{approx-poker} (row~(c)), so that we can compare them with the ICM (row~(a)), Monte Carlo (row (b)), and the regression approximation (row~(d)), which used \eqref{phat} (and its analogues for $\sigma=312$ and $\sigma=213$) with $A$, $B$, and $C$ as above and $N=1287$.

We find that linear interpolation and linear regression match to four or more decimal places, Monte Carlo to three, and ICM to one or two.  
\end{example}

\begin{table}[H]
\caption{\label{exact-regression}Two examples comparing the exact value of $P_{A,B,C}(\sigma)$ (to six significant digits) with its interpolation approximation $\bar P_{A,B,C}(\sigma)$ and its regression approximation $\hat P_{A,B,C}(\sigma)$.}
\tabcolsep=1.2mm 
\catcode`@=\active \def@{\hphantom{0}}
\begin{center}
\begin{tabular}{lllllll}
\hline
\noalign{\smallskip}
\phantom{000}$\sigma$ & @@123 & @@132 & @@213 & @@231 & @@312 & @@321  \\
\noalign{\smallskip}
\hline
\noalign{\smallskip}  
$P_{23,45,67}(\sigma)$      & 0.342769 & 0.264802 & 0.153527 & 0.108430 & 0.0685310 & 0.0619406 \\
\noalign{\smallskip}
$\bar P_{23,45,67}(\sigma)$ & 0.342763 & 0.264801 & 0.153533 & 0.108430 & 0.0685326 & 0.0619404 \\
\noalign{\smallskip}
$\hat P_{23,45,67}(\sigma)$ & 0.342744 & 0.264802 & 0.153552 & 0.108430 & 0.0685311 & 0.0619404 \\
\noalign{\smallskip}\hline
\noalign{\smallskip}  
$P_{10,40,90}(\sigma)$      & 0.532690 & 0.268542 & 0.110167 & 0.0553389 & 0.0171721 & 0.0160897 \\
\noalign{\smallskip}
$\bar P_{10,40,90}(\sigma)$ & 0.532702 & 0.268540 & 0.110155 & 0.0553369 & 0.0171744 & 0.0160917 \\
\noalign{\smallskip}
$\hat P_{10,40,90}(\sigma)$ & 0.532773 & 0.268542 & 0.110084 & 0.0553389 & 0.0171721 & 0.0160897 \\
\noalign{\smallskip}
\hline
\end{tabular}                           
\end{center}
\end{table}
\end{example} 

\begin{table}[H]
\caption{\label{approx-poker}Approximations to $P_{A,B,C}^{\text{GR}}(\sigma)$ when $A=169$, $B=301$, and $C=817$.  Row (a) uses ICM, row (b) uses Monte Carlo, row (c) uses linear interpolation, row (d) uses linear regression, and row (e) is exact (see the last paragraph of Subsection~\ref{subsec:exact}).  All figures are rounded to the degree shown.}
\catcode`@=\active \def@{\hphantom{0}}
\tabcolsep=.6mm 
\begin{center}
\begin{tabular}{lccccccc}
\hline
\noalign{\smallskip}
& $\sigma$ & 123 & 132 & 213 & 231 & 312 & 321  \\
\noalign{\smallskip}
\hline
\noalign{\smallskip}
(a) & $P_{A,B,C}^{\text{ICM}}(\sigma)$     & 0.406548@ & 0.193791@ & 0.228261@ & 0.095960@ & 0.0400865@ & 0.0353535@ \\
\noalign{\smallskip}
(b) & $\widehat P_{A,B,C}^{\text{GR}}(\sigma)$ & 0.419345@ & 0.207492@ & 0.215650@ & 0.106286@ & 0.0261205@ & 0.0251065@ \\
\noalign{\smallskip}
(c) & $\bar P_{A,B,C}^{\text{GR}}(\sigma)$ & 0.419603@ & 0.207878@ & 0.215207@ & 0.106174@  & 0.0259991@ & 0.0251395@ \\
\noalign{\smallskip}
(d) & $\hat P_{A,B,C}^{\text{GR}}(\sigma)$ & 0.419635@ & 0.207879@ & 0.215175@ & 0.106174@ & 0.0259984@ & 0.0251388@ \\
\noalign{\smallskip}
(e) & $P_{A,B,C}^{\text{GR}}(\sigma)$ & 0.4195973 & 0.2078788 & 0.2152123 & 0.1061744 & 0.02599843 & 0.02513876 \\
\noalign{\smallskip}
\hline
\end{tabular}
\end{center}
\end{table}

\newpage

\section{A conjecture and more than three players}\label{conjecture}

This section treats two further topics, the scaling conjecture and $k\ge4$ players (in particular, $k=4$).

\subsection{Scaling conjecture}\label{subsec:scaling}

The \textit{scaling conjecture} says, for all $A,B,C\ge1$, $\sigma\in S_3$, and $n\ge2$,
\begin{equation}\label{scaling-conj}
P_{nA,nB,nC}(\sigma)\doteq P_{A,B,C}(\sigma).
\end{equation}
As noted in Section~\ref{sec:intro}, this is closely related to the result, provable as a consequence of Donsker's theorem, that $\lim_{n\to\infty}P_{nA,nB,nC}(\sigma)$ exists and can be expressed in terms of standard two-dimensional Brownian motion. 

To formulate such a theorem, we adopt the setup used by Ferguson (1995).  Let $\Delta$ be the equilateral triangle with vertices $(-1,0)$, $(1,0)$, and $(0,\sqrt{3})$, and let $V_3$ be the edge that lies on the $x$-axis.  Let $A$, $B$, and $C$ be positive integers and $N:=A+B+C$.  Then the barycentric coordinates $(A/N,B/N,C/N)$ correspond to the initial state $\bm x:=((B-A)/N,\sqrt{3}\,C/N)$.

\begin{theorem}\label{Brownian-approx}
Let $\{\bm B(t),\;t\ge0\}$ be standard two-dimensional Brownian motion, and let $T_1$ be the exit time of $\bm x+\bm B$ from $\Delta$.  Then
\begin{equation}\label{Ferguson-limit}
\lim_{n\to\infty}[P_{nA,nB,nC}(321)+P_{nA,nB,nC}(312)]=P(\bm x+\bm B(T_1)\in V_3).
\end{equation}
Furthermore,
\begin{equation}\label{Hajek-limit}
\lim_{n\to\infty}P_{nA,nB,nC}(321)=E\bigg(\frac{|\bm x+\bm B(T_1)-(1,0)|}{2};\;\bm x+\bm B(T_1)\in V_3\bigg).
\end{equation}
\end{theorem}

The integrand in \eqref{Hajek-limit} is the proportion of the length of the edge $V_3$ that lies between the exit position $\bm x+\bm B(T_1)$ and the corner $(1,0)$ corresponding to player 2 winning all.  This amounts to applying the two-player gambler's ruin formula to the exit position.

Ferguson (1995) (see also Bruss, Louchard, and Turner, 2003) and Hajek (1987) used conformal mapping to give complicated expressions for the right sides of \eqref{Ferguson-limit} and \eqref{Hajek-limit}, respectively.  It remains to massage their formulas into computable form.  In a special case this can easily be done for Ferguson's formula.  If the initial state $(A,B,C)$ satisfies $A=B$, or equivalently, if the initial state in barycentric coordinates has the form $(a,a,1-2a)$, then the \textit{Mathematica} function defined in Figure~\ref{Ferguson-formula} gives the exit probability in \eqref{Ferguson-limit}.

\begin{figure}[htb]
\begin{center}
\includegraphics[width=4in]{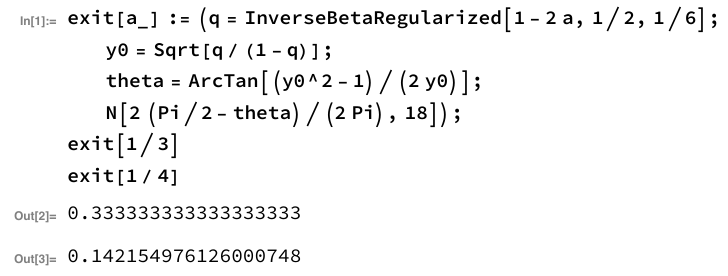}
\caption{\label{Ferguson-formula}\textit{Mathematica} code for the right side of \eqref{Ferguson-limit} when the initial state, in barycentric coordinates, is $(a,a,1-2a)$.}
\end{center}
\end{figure}

Ferguson (1995) showed that standard two-dimensional Brownian motion starting at $(0,\sqrt{3}/2)$ exits the equilateral triangle with vertices $(-1,0)$, $(1,0)$, and $(0,\sqrt{3})$ along the $x$-axis with probability about 0.1421.  Now $(0,\sqrt{3}/2)$ has barycentric coordinates $(\frac14,\frac14,\frac12)$, so Figure~\ref{Ferguson-formula} shows that Ferguson's probability, evaluated to 12 decimal places, is 0.142154976126.  On the other hand, the corresponding gambler's ruin probability with $N=300$ is
\begin{align}\label{Ferguson-approx}
&P_{75,75,150}(\text{player 3 goes broke first})\nonumber\\
&\qquad\qquad{}=P_{75,75,150}(312)+P_{75,75,150}(321)\doteq0.142154976161,
\end{align}
which we have computed to 18 decimal places, the first ten of which agree with Ferguson's number!

In support of the scaling conjecture we present evidence in Table~\ref{scaling}.  We have looked at many other examples.  Scaling to good approximation seems to hold always.

\begin{table}[htb]
\caption{\label{scaling}$P_{A,B,C}(\sigma)$ for $(A,B,C)=(2n,3n,5n)$ ($1\le n\le15$ and $n=20, 25, 30$), rounded to 12 significant digits, in support of the scaling conjecture.  Here we include only three choices of $\sigma$.  Results for the others can be deduced from \eqref{mart-identities}.}
\tabcolsep=2mm 
\catcode`@=\active \def@{\hphantom{0}}
\begin{center}
\begin{tabular}{lcccccc}
\hline
\noalign{\smallskip}
$A,B,C$ & $\sigma=213$ & $\sigma=312$ & $\sigma=321$ \\
\noalign{\smallskip}\hline
\noalign{\smallskip}
$2,3,5$    & 0.190419015064 & 0.0704242611225 & 0.0662121426098 \\
$4,6,10$   & 0.190374670083 & 0.0704067672263 & 0.0662043067857 \\
$6,9,15$   & 0.190371967724 & 0.0704057817695 & 0.0662038677034 \\
$8,12,20$  & 0.190371502992 & 0.0704056143412 & 0.0662037932082 \\
$10,15,25$ & 0.190371375036 & 0.0704055684270 & 0.0662037727906 \\
$12,18,30$ & 0.190371328913 & 0.0704055519070 & 0.0662037654463 \\
$14,21,35$ & 0.190371309103 & 0.0704055448186 & 0.0662037622955 \\
$16,24,40$ & 0.190371299477 & 0.0704055413763 & 0.0662037607656 \\
$18,27,45$ & 0.190371294349 & 0.0704055395436 & 0.0662037599511 \\
$20,30,50$ & 0.190371291418 & 0.0704055384960 & 0.0662037594855 \\
$22,33,55$ & 0.190371289645 & 0.0704055378624 & 0.0662037592040 \\
$24,36,60$ & 0.190371288521 & 0.0704055374610 & 0.0662037590256 \\
$26,39,65$ & 0.190371287781 & 0.0704055371968 & 0.0662037589082 \\
$28,42,70$ & 0.190371287279 & 0.0704055370172 & 0.0662037588284 \\
$30,45,75$ & 0.190371286927 & 0.0704055368917 & 0.0662037587726 \\
\noalign{\smallskip}            
$40,60,100$& 0.190371286171 & 0.0704055366216 & 0.0662037586526 \\
$50,75,125$& 0.190371285964 & 0.0704055365478 & 0.0662037586198 \\
$60,90,150$& 0.190371285890 & 0.0704055365213 & 0.0662037586080 \\
\noalign{\smallskip}\hline
\noalign{\smallskip}
\end{tabular}
\end{center}
\end{table}

A second piece of evidence comes from
\begin{equation}\label{exact P(123)}
P_{i,i,N-2i}(123)=P_{i,i,N-2i}(213)=\frac12\bigg(1-\frac{2i}{N}\bigg),\qquad 1\le i<N/2.
\end{equation}
These are \textit{exactly} invariant under scaling.  Indeed, they match the ICM.

A third piece of evidence comes from the Brownian motion approximation of the random walk.  As we have already seen for $k=3$, the gambler's ruin walk converges to Brownian motion on the $k$-simplex (Denisov and Wachtel, 2015).  It follows that the first hitting probabilities converge to those of Brownian motion.  Finally, the Brownian motion extinction probabilities are scale invariant via properties of Brownian motion. 

A fourth piece of evidence comes from the asymptotic approximation \eqref{D,HE,SC-eq} above.  This is (approximately) scale invariant.

The rapid convergence of rescaled probabilities (as seen in Table~\ref{scaling}) is surprising.  Theorem~\ref{Brownian-approx} shows that these approach limits expressible in terms of standard two-dimensional Brownian motion.  We might denote the limit of $P_{nA,nB,nC}(\sigma)$ as $n\to\infty$ by $P_{A,B,C}^{\text{BM}}(\sigma)$, where the superscript refers to Brownian motion.  For example, if $\sigma=321$, this limit is given by \eqref{Hajek-limit}.  Usually, Gaussian approximation of features of random walk converge at rate $1/\sqrt{N}$.  The numerics would be explained by the following conjecture, which may be regarded as a more precise version of the scaling conjecture \eqref{scaling-conj}.

\begin{conjecture}
$(a)$ For each $A,B,C\ge1$, $\sigma\in S_3$, and $n\ge2$,
\begin{equation*}
|P_{nA,nB,nC}(\sigma)-P_{A,B,C}(\sigma)|<0.0004.
\end{equation*}
$(b)$ For each $A,B,C\ge1$ and $\sigma\in S_3$, let $N:=A+B+C$.  Then
\begin{equation*}
|P_{nA,nB,nC}(\sigma)-P_{A,B,C}^{{\rm BM}}(\sigma)|=O\bigg(\frac{1}{(nN)^4}\bigg)\quad\text{as }n\to\infty.
\end{equation*}
\end{conjecture}

\noindent In the case of (a), we have found differences as large as 0.000383.  As for (b), the ten-digit match seen in \eqref{Ferguson-approx} is consistent with this because $1/(300)^4\doteq1.235\times10^{-10}$.

In practical problems scale invariance and smoothness (so fine details don't matter much) can reduce things to ``manageable numbers'' within the range of computer calculation.

\subsection{Gambler's ruin with more than three players}

The questions above make sense for $k$ players with initial capitals $A_1,A_2,\ldots,A_k$.  The exact calculations of Subsection~\ref{subsec:exact} are (potentially) available.  We have carried them out to give exact results for $k=4$ and $N:=A_1+A_2+A_3+A_4$ as large as 100.  The results for $N=100$ are in the supplementary materials (Section~\ref{suppl}).  By analogy with \eqref{exact P(123)},
\begin{equation*}
P_{i,i,i,N-3i}(1234)=\frac16\bigg(1-\frac{3i}{N}\bigg),\qquad 1\le i<N/3.
\end{equation*}

The scaling conjecture for $k=4$, either in the form
\begin{equation*}
P_{A',B',C',D'}(\sigma)\doteq P_{A,B,C,D}(\sigma)\quad\text{whenever}\quad\frac{A'}{A}=\frac{B'}{B}=\frac{C'}{C}=\frac{D'}{D},
\end{equation*}
or in the equivalent form
\begin{equation*}
P_{nA,nB,nC,nD}(\sigma)\doteq P_{A,B,C,D}(\sigma),\qquad n\ge2,
\end{equation*}
seems to hold. Here $A,B,C,D\ge1$ and $\sigma\in S_4$ are arbitrary.  Table~\ref{scaling-4players} gives a few data points.  These numbers are consistent with those of Marfil and David (2020).  Notice that convergence is slower for four players than for three.

\begin{table}[htb]
\caption{\label{scaling-4players}$P_{A,B,C,D}(\sigma)$ for $(A,B,C,D)=(n,2n,3n,4n)$ ($1\le n\le 10$) and four choices of $\sigma\in S_4$, rounded to nine decimal places, in support of the scaling conjecture.  Notice that the rate of convergence appears to be slower than for the three-player data in Table~\ref{scaling}.}
\tabcolsep=2mm 
\catcode`@=\active \def@{\hphantom{0}}
\begin{small}
\begin{center}
\begin{tabular}{lcccc}
\hline
\noalign{\smallskip}
$A,B,C,D$ & $\sigma=1234$ & $\sigma=2143$ & $\sigma=3412$ & $\sigma=4321$ \\
\noalign{\smallskip}\hline
\noalign{\smallskip}
$1,2,3,4$    & 0.147755766 & 0.055231830 & 0.012087939 & 0.007499579 \\
$2,4,6,8$    & 0.148462055 & 0.054618468 & 0.012147611 & 0.007459339 \\
$3,6,9,12$   & 0.148582024 & 0.054511807 & 0.012158593 & 0.007452294 \\
$4,8,12,16$  & 0.148621208 & 0.054476628 & 0.012162415 & 0.007449874 \\
$5,10,15,20$ & 0.148638685 & 0.054460859 & 0.012164179 & 0.007448762 \\
$6,12,18,24$ & 0.148647981 & 0.054452450 & 0.012165136 & 0.007448161 \\
$7,14,21,28$ & 0.148653514 & 0.054447436 & 0.012165712 & 0.007447800 \\
$8,16,24,32$ & 0.148657074 & 0.054444206 & 0.012166086 & 0.007447565 \\
$9,18,27,36$ & 0.148659501 & 0.054442002 & 0.012166342 & 0.007447405 \\
$10,20,30,40$& 0.148661229 & 0.054440432 & 0.012166526 & 0.007447290 \\
\noalign{\smallskip}\hline
\noalign{\smallskip}
\end{tabular}
\end{center}
\end{small}
\end{table}

The ICM formula is available for all $k$.  Preliminary investigations (including Table~\ref{wsop-interp4}) suggest it is just as unreliable as an approximation to $P_{A_1,\ldots,A_k}(\sigma)$ as it is when $k=3$.  We have tried interpolation (Subsection~\ref{subsec:interpolation4}) but not yet regression.  Plots such as Figure~\ref{RatioPlot} are not feasible when $k=4$.

One final point:  The $\text{constant}/N^3$ results described above for $k=3$ should not stir false hope of similar results for $k=4$.  There are reasons to expect that
\begin{equation*}
P_{1,1,1,N-3}(4321)\sim\frac{\text{constant}}{N^\kappa}
\end{equation*}
with $\kappa$ an irrational number.  This (heuristically) follows from the connection between gambler's ruin and the ``cops and robbers'' problem.  See Ratzkin and Treibergs (2009).  Table~\ref{4-player-data} gives ten data points, which suggest $\kappa=5.72\cdots$.

\begin{table}[H]
\caption{\label{4-player-data}$P_{1,1,1,N-3}(4321)$ for $N=10,20,\ldots,100$.}
\tabcolsep=2mm 
\catcode`@=\active \def@{\hphantom{0}}
\begin{center}
\begin{tabular}{rlcrl}
\hline
\noalign{\smallskip}
$N$ & $P_{1,1,1,N-3}(4321)$ &@@& $N$ & $P_{1,1,1,N-3}(4321)$ \\
\noalign{\smallskip}\hline
\noalign{\smallskip}
10 & $2.61956573\times10^{-4}$ && @60 & $9.43556904\times10^{-9}$ \\
20 & $5.03729359\times10^{-6}$ && @70 & $3.90711745\times10^{-9}$ \\
30 & $4.96691782\times10^{-7}$ && @80 & $1.82032195\times10^{-9}$ \\
40 & $9.58966829\times10^{-8}$ && @90 & $9.28008330\times10^{-10}$ \\
50 & $2.67684672\times10^{-8}$ && 100& $5.07937120\times10^{-10}$ \\
\noalign{\smallskip}\hline
\noalign{\smallskip}
\end{tabular}
\end{center}
\end{table}

\subsection{Linear interpolation for four players}\label{subsec:interpolation4}

Just as we could interpolate three-player elimination order probabilities with arbitrary $N$ from three known such probabilities with $N=300$, we can also interpolate four-player elimination order probabilities with arbitrary $N$ from four known such probabilities with $N=100$.

Given positive integers $A$, $B$, $C$, and $D$, let $N:=A+B+C+D$ and
\begin{equation*}
A_0:=A\,\frac{100}{N},\quad B_0:=B\,\frac{100}{N},\quad C_0:=C\,\frac{100}{N},\quad D_0:=D\,\frac{100}{N}.
\end{equation*}
Typically, these are not integers.  Therefore, consider the eight points
\begin{align*}
\bm v_{000}&:=(\lfloor A_0\rfloor,\lfloor B_0\rfloor,\lfloor C_0\rfloor, 100-\lfloor A_0\rfloor-\lfloor B_0\rfloor-\lfloor C_0\rfloor),\\
\bm v_{001}&:=(\lfloor A_0\rfloor,\lfloor B_0\rfloor,\lceil C_0\rceil,100-\lfloor A_0\rfloor-\lfloor B_0\rfloor-\lceil C_0\rceil),\\
\bm v_{010}&:=(\lfloor A_0\rfloor,\lceil B_0\rceil,\lfloor C_0\rfloor,100-\lfloor A_0\rfloor-\lceil B_0\rceil-\lfloor C_0\rfloor),\\
\bm v_{100}&:=(\lceil A_0\rceil,\lfloor B_0\rfloor,\lfloor C_0\rfloor,100-\lceil A_0\rceil-\lfloor B_0\rfloor-\lfloor C_0\rfloor),\\
\bm v_{011}&:=(\lfloor A_0\rfloor,\lceil B_0\rceil,\lceil C_0\rceil,100-\lfloor A_0\rfloor-\lceil B_0\rceil-\lceil C_0\rceil),\\
\bm v_{101}&:=(\lceil A_0\rceil,\lfloor B_0\rfloor,\lceil C_0\rceil,100-\lceil A_0\rceil-\lfloor B_0\rfloor-\lceil C_0\rceil),\\
\bm v_{110}&:=(\lceil A_0\rceil,\lceil B_0\rceil,\lfloor C_0\rfloor,100-\lceil A_0\rceil-\lceil B_0\rceil-\lfloor C_0\rfloor),\\
\bm v_{111}&:=(\lceil A_0\rceil,\lceil B_0\rceil,\lceil C_0\rceil,100-\lceil A_0\rceil-\lceil B_0\rceil-\lceil C_0\rceil),
\end{align*}
and choose four of them for the purpose of linear interpolation, discarding any whose fourth coordinate is neither $\lfloor D_0\rfloor$ nor $\lceil D_0\rceil$.  Denote by $\{a\}:=a-\lfloor a\rfloor$ the fractional part of $a$.

If $\{A_0\}+\{B_0\}+\{C_0\}\in(0,1)$, then we choose $\bm v_{000}$, $\bm v_{001}$, $\bm v_{010}$, and $\bm v_{100}$.

If $\{A_0\}+\{B_0\}+\{C_0\}\in(2,3)$, then we choose $\bm v_{011}$, $\bm v_{101}$, $\bm v_{110}$, and $\bm v_{111}$.

If $\{A_0\}+\{B_0\}+\{C_0\}\in(1,2)$, then we choose four of the six points $\bm v_{001}$, $\bm v_{010}$, $\bm v_{100}$, $\bm v_{011}$, $\bm v_{101}$, and $\bm v_{110}$ in such a way that the resulting tetrahedron contains $(A_0,B_0,C_0,D_0)$ in its interior.  The choice is not unique.

Let us call these four points $(A_i,B_i,C_i,D_i)$ ($i=1,2,3,4$).  We can estimate $P_{A,B,C,D}(\sigma)$ by linear interpolation from the four values of $P_{A_i,B_i,C_i,D_i}(\sigma)$ ($i=1,2,3,4$).  As before, we represent $(A_0,B_0,C_0,D_0)$ in barycentric coordinates.  The relevant weights are
\begin{equation*}
\setlength{\arraycolsep}{2mm}
\begin{pmatrix}\lambda_1\\ \lambda_2\\ \lambda_3\end{pmatrix}:=\begin{pmatrix}A_1-A_4&A_2-A_4&A_3-A_4\\B_1-B_4&B_2-B_4&B_3-B_4\\
C_1-C_4&C_2-C_4&C_3-C_4\end{pmatrix}^{-1}\begin{pmatrix}A_0-A_4\\ B_0-B_4\\ C_0-C_4\end{pmatrix}
\end{equation*}
and $\lambda_4:=1-\lambda_1-\lambda_2-\lambda_3$, so that
\begin{equation*}
(A_0,B_0,C_0,D_0)=\sum_{i=1}^4\lambda_i(A_i,B_i,C_i,D_i),
\end{equation*}
and our interpolation estimate is then
\begin{equation*}
\bar P_{A,B,C,D}(\sigma):=\sum_{i=1}^4\lambda_i P_{A_i,B_i,C_i,D_i}(\sigma).
\end{equation*}
If one or more of the weights $\lambda_i$ is negative, that indicates $(A_0,B_0,C_0,D_0)$ lies outside the resulting tetrahedron, and we must choose the four points differently.

\begin{example}
At the final table of the 2019 World Series of Poker Millionaire Maker Event, at the time the fifth-place finisher was eliminated, the remaining four players had chip counts (in units of $100{,}000$, 1/16 of the big blind) equal to $A=97$, $B=125$, $C=144$, and $D=1839$ (WSOP, 2019b).  See Table~\ref{2019wsop2}.  Thus, $N=2205$ and $A$, $B$, $C$, and $D$, multiplied by $100/N$, are $A_0\doteq4.40$, $B_0\doteq5.67$, $C_0\doteq6.53$, and $D_0\doteq83.40$.  Since $\{A_0\}+\{B_0\}+\{C_0\}\doteq1.60$, we must choose four of the six vertices $\bm v_2=(4,5,7,84)$, $\bm v_3=(4,6,6,84)$, $\bm v_4=(5,5,6,84)$, $\bm v_5=(4,6,7,83)$, $\bm v_6=(5,5,7,83)$, and $\bm v_7=(5,6,6,83)$.  We choose $\bm v_2$, $\bm v_3$, $\bm v_5$, and $\bm v_7$, the four points closest to $(A_0,B_0,C_0,D_0)$.  We find that
\begin{equation*}
\lambda_1=\frac{146}{441},\quad\lambda_2=\frac{31}{441},\quad\lambda_3=\frac{88}{441},\quad\lambda_4=\frac{176}{441}, 
\end{equation*}
and results are shown in Table~\ref{wsop-interp4}.  For the record, the actual elimination order turned out to be $\sigma=1243$, the seventh most likely result.  

\begin{table}[H]
\caption{\label{2019wsop2}The final four in the 2019 World Series of Poker Millionaire Maker Event.  The big blind was $1{,}600{,}000$.}
\catcode`@=\active \def@{\hphantom{0}}
\tabcolsep=1.5mm 
\begin{center}
\begin{tabular}{lcccr}
\hline\noalign{\smallskip}
@@@@player & chip count & big blinds & big blinds & actual payoff \\
           &            & (rounded)  &  ${}\times16$ (exact) &             \\
\noalign{\smallskip}\hline
\noalign{\smallskip}
Vincas Tamasauskas & @@$9{,}700{,}000$ & @@6 & @@97    & $\$464{,}375$ \\
Lokesh Garg        & @$12{,}500{,}000$ & @@8 & @125   & $\$619{,}017$ \\
John Gorsuch       & @$14{,}400{,}000$ & @@9 & @144   & $\$1{,}344{,}930$ \\
Kazuki Ikeuchi     & $183{,}900{,}000$ & 115 & 1839  & $\$830{,}783$ \\
\noalign{\smallskip}\hline
\noalign{\smallskip}
@@@@totals          & $220{,}500{,}000$ & 138 & 2205 &\\
\noalign{\smallskip}\hline
\end{tabular}
\end{center}
\end{table}

\begin{table}[htb]
\caption{\label{wsop-interp4}For $(A,B,C,D)=(97,125,144,1839)$, row (a) gives $P_{A,B,C,D}^{\text{ICM}}(\sigma)$, and row (b) gives
the interpolated approximations $\bar P_{A,B,C,D}^{\text{GR}}(\sigma)$.  Here $\eps=10^{-5}$.}
\tabcolsep=1mm 
\catcode`@=\active \def@{\hphantom{0}}
\begin{small}
\begin{center}
\begin{tabular}{llllllll}
\hline
\noalign{\smallskip}
& $\sigma=1234$ & $\sigma=1243$ & $\sigma=1324$ & $\sigma=1342$ & $\sigma=1423$ & $\sigma=1432$ \\
& $\sigma=2134$ & $\sigma=2143$ & $\sigma=2314$ & $\sigma=2341$ & $\sigma=2413$ & $\sigma=2431$ \\
& $\sigma=3124$ & $\sigma=3142$ & $\sigma=3214$ & $\sigma=3241$ & $\sigma=3412$ & $\sigma=3421$ \\
& $\sigma=4123$ & $\sigma=4132$ & $\sigma=4213$ & $\sigma=4231$ & $\sigma=4312$ & $\sigma=4321$ \\
\noalign{\smallskip}\hline
\noalign{\smallskip}
(a) & 0.184762 & 0.0328106 & 0.170195 & 0.0299478 & 0.00376238 & 0.00372801 \\
& 0.143375 & 0.0254611 & 0.118324 & 0.0205440 & 0.00287798 & 0.00281381 \\
& 0.114645 & 0.0201732 & 0.102712 & 0.0178333 & 0.00245171 & 0.00241914 \\
& 0.000198451 & 0.000196638 & 0.000195621 & 0.00019126 & 0.000191977 & 0.000189427 \\
\noalign{\smallskip}\hline
\noalign{\smallskip}
(b) & 0.193685 & 0.0365869 & 0.177240 & 0.0338414 & 0.00127429 & 0.00127379 \\
& 0.139180 & 0.0264827 & 0.118035 & 0.0228907 & 0.000947236 & 0.000946428 \\
& 0.107017 & 0.0207477 & 0.0988571 & 0.0193277 & 0.000811526 & 0.000811141 \\
& 0.751143$\,\eps$ & 0.750991$\,\eps$ & 0.750703$\,\eps$ & 0.750331$\,\eps$ & 0.750250$\,\eps$ & 0.750030$\,\eps$ \\
\noalign{\smallskip}\hline
\noalign{\smallskip}
\end{tabular}
\end{center}
\end{small}
\end{table}
\end{example}

\section{Summary}\label{sec:summary}

In Summary, we have discussed six different methods of approximating the gambler's ruin probabilities:
\begin{enumerate}
\item Exact computation by Markov chain methods (Subsection~\ref{subsec:exact})
\item Arbitrarily precise computation by Jacobi iteration (Subsection~\ref{subsec:iteration})
\item Linear interpolation from exact probabilities (Subsection~\ref{subsec:interpolation})
\item Monte Carlo methods (Subsection~\ref{subsec:Monte})
\item Regression on ICM (Section~\ref{sec:regression})
\item Approximation by Brownian motion (Subsection~\ref{subsec:scaling})
\end{enumerate}

While exact computation using Markov chain methods and arbitrarily precise computation using Jacobi iteration are feasible for $N:=A+B+C$ not too large, it seems difficult for $N$ of practical interest. Linear interpolation is our preferred method, using the nearly exact results for $N=300$. Monte Carlo allows computation for a single $A,B,C$ of interest and is useful for two- or three-digit accuracy.  Regression analysis is quite accurate but probably needs an app to to be ``real-time useful.'' Brownian approximation changes the problem into one that requires special function calculations and so probably also needs an app.  Finally, the widely used ICM is roughly useful (say for single-digit accuracy), and it can be ``done in your head.''

\section{Supplementary materials}\label{suppl}

Supplementary materials include four \textit{Mathematica} programs, four output files, and three regression analyses using \textit{Mathematica}.  Here are the details.

\begin{raggedright}
\begin{enumerate}
\item \textit{Mathematica} program to compute three-player elimination order probabilities in double precision by Markov chain methods ($N$ arbitrary), and output when $N=200$.  \url{http://www.math.utah.edu/~ethier/3ruin-program.nb}
\url{http://www.math.utah.edu/~ethier/3ruin200-output}
\item \textit{Mathematica} program to compute three-player elimination order probabilities in double precision by iteration ($N$ arbitrary), and output when $N=300$.  \url{http://www.math.utah.edu/~ethier/3iteration-program.nb}
\url{http://www.math.utah.edu/~ethier/3iteration300-output}
\item \textit{Mathematica} program to compute four-player elimination order probabilities in single precision by Markov chain methods ($N$ arbitrary), and output when $N=50$.   \url{http://www.math.utah.edu/~ethier/4ruin-program.nb}
\url{http://www.math.utah.edu/~ethier/4ruin50-output}
\item \textit{Mathematica} program to compute four-player elimination order probabilities in single precision by iteration ($N$ arbitrary), and output when $N=100$.  \url{http://www.math.utah.edu/~ethier/4iteration-program.nb}
\url{http://www.math.utah.edu/~ethier/4iteration100-output}
\item \textit{Mathematica} files containing regression analyses for $\sigma=321$, 312, and 213.   
\url{http://www.math.utah.edu/~ethier/regression321-300.nb}
\url{http://www.math.utah.edu/~ethier/regression312-300.nb}
\url{http://www.math.utah.edu/~ethier/regression213-300.nb}
\end{enumerate}
\end{raggedright}

\begin{newreferences}

\item Aguilar, J. (2016) How to negotiate final table deals like a pro.  Upswing Poker. \url{https://upswingpoker.com/final-table-deal-making-tournaments/}.

\item Aldous, D., Lanoue, D. and Salez, J. (2015) The compulsive gambler process.  \textit{Electron. J. Probab.} \textbf{20}, 1--18.

\item Bachelier, L. (1912) \textit{Calcul des Probabilit\'es, Vol.~1}.  Gauthier-Villars, Paris.

\item Bruss, F. T., Louchard, G., and Turner, J. W. (2003) On the $N$-tower problem and related problems. \textit{Adv.\ Appl.\ Probab.} \textbf{35} (1), 278--294.

\item Cover, T. M. (1987) Gambler’s ruin: A random walk on the simplex. In \textit{Open Problems in Communication and Computation.} (T. M. Cover and B. Gopinath, eds.) 155. Springer, New York.

\item David, G. (2015) Markov chain solution to the 3-tower problem. In \textit{Information and Communication Technology.} (I. Khalil, E. Neuhold, A. Tjoa, L. Xu, and I. You, eds.)  ICT-EurAsia 2015. Lecture Notes in Computer Science \textbf{9357}, 121--128. Springer, Cham. 

\item Davis, T. A. (2006).  \textit{Direct Methods for Sparse Linear Systems}.  SIAM, Philadelphia.

\item Denisov, D. and Wachtel, V. (2015) Random walks in cones.  \textit{Ann.\ Probab.} \textbf{43} (3), 992--1044.

\item Diaconis, P. (1988) \textit{Group Representations in Probability and Statistics}.  Lecture Notes--Monograph Series \textbf{11}.  Institute of Mathematical Statistics, Hayward, CA.

\item Diaconis, P. and Freedman, D. (1979) On rounding percentages.  \textit{J. Amer. Statist. Assoc.} \textbf{74} (366a), 359--364.

\item Diaconis, P., Houston-Edwards, K., and Saloff-Coste, L. (2021) Gambler's ruin estimates on finite inner uniform domains.  \textit{Ann.\ Appl.\ Probab.}, \textbf{31} (2), 865--895.

\item Engel, A. (1993) The computer solves the three tower problem. \textit{Amer.\ Math.\ Monthly} \textbf{100} (1), 62--64.

\item Ethier, S. N. (2010) \textit{The Doctrine of Chances: Probabilistic Aspects of Gambling}.   Springer, Berlin and Heidelberg.

\item Feller, W. (1968) \textit{An Introduction to Probability Theory and Its Applications, Volume I}, Third Edition.  John Wiley \& Sons, Inc., New York.

\item Ferguson, T. (1995) Gambler's ruin in three dimensions.  Unpublished.  \url{https://www.math.ucla.edu/~tom/papers/unpublished/gamblersruin.pdf}.

\item Ganzfried, S. and Sandholm, T. (2008) Computing an approximate jam/fold equilibrium for 3-player
no-limit Texas hold’em tournaments.  \textit{AAMAS08: 7th International Conference on Autonomous Agents and Multi Agent Systems} (L. Padgham, D. Parkes, J. M\"{u}ller, and S. Parsons, eds.) 919--926. International Foundation for Autonomous Agents and Multiagent Systems, Richland, SC.

\item Gilbert, G. T. (2009) The independent chip model and risk aversion.  \url{https://arxiv.org/abs/0911.3100}.

\item Gilliland, D., Levental, S., and Xiao, Y. (2007) A note on absorption probabilities in one-dimensional random walk via complex-valued martingales. \textit{Statist.\ Probab.\ Lett.} \textbf{77} (11), 1098--1105.

\item Hajek, B. (1987) Gambler’s ruin: A random walk on the simplex.  In \textit{Open Problems in Communication and Computation.} (T. M. Cover and B. Gopinath, eds.) 204--207.  Springer, New York.

\item Harville, D. A. (1973) Assigning probabilities to the outcomes of multi-entry competitions.  \textit{J. Amer.\ Statist.\ Assoc.} \textbf{68} (342), 312--316.

\item ICMizer (2020) Poker ICM calculator for final table deals. \url{https://www.icmpoker.com/icmcalculator/}.

\item Kemeny, J. G. and Snell, J. L. (1976) \textit{Finite Markov Chains}. Springer-Verlag, New York.

\item Kim, M. S. (2005) Gambler's ruin in many dimensions and optimal strategy in repeated multi-player games with application to poker.  Masters thesis, UCLA.

\item Luce, R. D. (1959) \textit{Individual Choice Behavior: A Theoretical Analysis}.  Wiley, New York.

\item Luce, R. D. (1977) The choice axiom after twenty years. \textit{J. Math.\ Psych.} \textbf{15} (3), 215--233.

\item Malmuth, M. (1987) \textit{Gambling Theory and Other Topics}. Two Plus Two Publishing, Henderson, NV.

\item Malmuth, M. (2004) \textit{Gambling Theory and Other Topics}, Sixth Edition. Two Plus Two Publishing, Henderson, NV.  

\item Marfil, R. I. D. and David, G. (2020) On the placing probabilities for the four-tower problem using recursions based on multigraphs.  \textit{J. Math. Soc. Philippines} \textbf{43} (1), 19--32.

\item Plackett, R. L. (1975) The analysis of permutations.  \textit{J. R. Statist.\ Soc., Ser. C (Appl.\ Statist.)} \textbf{24} (2), 193--202.

\item Ratzkin, J., and Treibergs, A. (2009) A capture problem in Brownian motion and eigenvalues of spherical domains.  \textit{Trans.\ Amer.\ Math.\ Soc.} \textbf{361} (1), 391--405.

\item Ross, S. M. (2009) A simple solution to a multiple player gambler's ruin problem.  \textit{Amer.\ Math.\ Monthly} \textbf{116} (1), 77--81.

\item Song, S. and Song, J. (2013) A note on the history of the gambler's ruin problem.  \textit{Comm.\ Statist.\ Applic.\ Methods} \textbf{20} (1), 1--12.

\item Stern, H. S. (2008) Estimating the probabilities of the outcomes of a horse race (alternatives to the Harville formulas).  In \textit{Efficiency of Racetrack Betting Markets}, 2008 Edition (D. B. Hausch, V. S. Y. Lo, and W. T. Ziemba, eds.) 225--235.  World Scientific Publishing, Singapore.

\item Stirzaker, D. (1994) Tower problems and martingales. \textit{Math.\ Scientist} \textbf{19} (1), 52--59.

\item Stirzaker, D. (2006) Three-handed gambler's ruin. \textit{Adv.\ Appl.\ Probab.} \textbf{38} (1), 284--286.

\item Swan, Y. C. and Bruss, F. T. (2006) A matrix-analytic approach to the $N$-player ruin problem.  \textit{J. Appl. Probab.} \textbf{43} (3), 755--766.

\item Turner, H., van Etten, J., Firth, D., and Kosmidis, I. (2017) Introduction to PlackettLuce.  Microsoft R Application Network.  \url{https://mran.microsoft.com/snapshot/2017-12-15/web/packages/PlackettLuce/vignettes/Overview.html}

\item WSOP (2019a) Hossein Ensan wins the 2019 WSOP Main Event (\$10,000,000).  \url{https://www.wsop.com/tournaments/updates/?aid=2&grid=1622&tid=17298&dayof=7661&rr=5&curpage=4}.

\item WSOP (2019b) John Gorsuch completes epic comeback to win 2019 WSOP Millionaire Maker for \$1,344,930.  \url{https://www.wsop.com/tournaments/updates/?aid=2&grid=1622&tid=17287&dayof=7470&rr=5}.

\end{newreferences}

\end{document}